\documentclass[a4paper,leqno]{amsart}

\usepackage{latexsym}
\usepackage[english]{babel}
\usepackage{fancyhdr}
\usepackage[mathscr]{eucal}
\usepackage{amsmath}
\usepackage{mathrsfs}
\usepackage{amsthm}
\usepackage{amsfonts}
\usepackage{amssymb}
\usepackage{amscd}
\usepackage{bbm}
\usepackage{graphicx}
\usepackage{subcaption}
\usepackage{graphics}
\usepackage{latexsym}
\usepackage{color}

\usepackage{pifont}
\usepackage{booktabs} 

\newcommand{\ud}{\mathrm{d}}

\newcommand{\cH}{\mathcal{H}}
\newcommand{\ran}{\mathrm{ran}}

\newcommand{\N}{\mathbb N}

\theoremstyle{plain}
\newtheorem{theorem}{Theorem}[section]
\newtheorem{lemma}[theorem]{Lemma}
\newtheorem{corollary}[theorem]{Corollary}
\newtheorem{proposition}[theorem]{Proposition}

\theoremstyle{definition}

\newtheorem{remark}[theorem]{Remark}

\numberwithin{equation}{section}

\begin{document}

\title[Convergence of conjugate gradient with unbounded operators]
{Convergence of the conjugate gradient method with unbounded operators}
\author[N.~Caruso]{Noe Caruso}
\address[N.~Caruso]{International School for Advanced Studies -- SISSA \\ via Bonomea 265 \\ I-34136 Trieste (ITALY), and Gran Sasso Science Institute -- GSSI \\  Viale Francesco Crispi 7 \\ I-67100 L'Aquila (ITALY).}
\email{noe.caruso@gssi.it}
\author[A.~Michelangeli]{Alessandro Michelangeli}
\address[A.~Michelangeli]{Institute for Applied Mathematics \\ and Hausdorff Center for Mathematics, University of Bonn \\ Endenicher Allee 60 \\ 
D-53115 Bonn (GERMANY).}
\email{michelangeli@iam.uni-bonn.de}


\begin{abstract}
 In the framework of inverse linear problems on infinite-dimensional Hilbert space, we prove the convergence of the conjugate gradient iterates to an exact solution to the inverse problem in the most general case where the self-adjoint, non-negative operator is unbounded and with minimal, technically unavoidable assumptions on the initial guess of the iterative algorithm. The convergence is proved to always hold in the Hilbert space norm (error convergence), as well as at other levels of regularity (energy norm, residual, etc.) depending on the regularity of the iterates. We also discuss, both analytically and through a selection of numerical tests, the main features and differences of our convergence result as compared to the case, already available in the literature, where the operator is bounded. 
\end{abstract}

\date{\today}
\makeatletter{\renewcommand*{\@makefnmark}{}
\footnotetext{Operators and Matrices.}\makeatother}



\thanks{Partially supported by the Alexander von Humboldt Foundation.}

\maketitle


\section{Introduction}\label{intro}

We are concerned in this work with the rigorous proof of the convergence, in various meaningful senses, of a particular and well-known iterative algorithm for solving inverse linear problems, the celebrated conjugate gradient method, in the generalised setting of \emph{unbounded} operators on Hilbert space.

In abstract terms, given a Hilbert space $\cH$ over the (real or) complex field and a non-negative self-adjoint operator $A$ on $\cH$, we consider the \emph{inverse linear problem}
\begin{equation}\label{eq:invLP}
 Af\,=\,g\,,\qquad g\in\ran A
\end{equation}
in the unknown $f\in\cH$ with datum $g$ -- assuming $g\in\ran A$ makes the problem \eqref{eq:invLP} solvable. $A$ is allowed to be unbounded, in which case necessarily $\cH$ has infinite dimension and the domain $\mathcal{D}(A)$ of $A$ is only a dense subspace of $\cH$. The positivity assumption on $A$ reads $\langle \psi,A\psi\rangle\geqslant 0$ for all $\psi\in\mathcal{D}(A)$: here and in the following $\langle\,\cdot\,,\,\cdot\rangle$ is the scalar product in $\cH$: if $\cH$ is taken over the complex field, then $\langle\,\cdot\,,\,\cdot\rangle$ is assumed to be anti-linear in the first entry and linear in the second, and $\|\cdot\|$ is the corresponding norm. For a positive, self-adjoint operator $A$ we shall also use the customary notation $A=A^*\geqslant\mathbb{O}$.

This setting generalises the classical, \emph{finite-dimensional} one where $\cH=\mathbb{C}^d$ for some $d\in\mathbb{N}$ and $A$ is a $d\times d$ positive semi-definite matrix (in which case \eqref{eq:invLP} can be interpreted as a system of $d$ linear equations), as well as the setting where $A$ is a bounded (and everywhere-defined) self-adjoint operator on an \emph{infinite-dimensional} Hilbert space.

Infinite dimensionality is also the framework where the phenomenon of ill-posed\-ness may occur. Indeed, it is a standard fact that for a (not necessarily bounded) self-adjoint operator $A$ on an infinite-dimensional Hilbert space $\cH$ the properties
\begin{itemize}
 \item[(i)] the point 0 belongs to $\sigma(A)$ and is not isolated in $\sigma(A)$,
 \item[(ii)] $\ran A$ is not closed,
 \item[(iii)] on $\ran A$ the operator $A$ has unbounded inverse,
\end{itemize}
are all equivalent (and none could occur if $\dim\cH<+\infty$): when any such property is satisfied, the solution $f$ cannot depend continuously on the datum $g$, as is evident from (iii), and the problem \eqref{eq:invLP} is said to be \emph{ill-posed}.

As opposite to that, if any among (i), (ii), and (iii) above fails to hold and in addition $A$ is injective, and hence equivalently if $A$ has an everywhere-defined bounded inverse, the problem \eqref{eq:invLP} is \emph{well-posed}: in this case the solution $f$ exists, is unique, and depends continuously on the datum $g$.

A popular algorithm for the numerical solution to \eqref{eq:invLP} in the above-mentioned classical framework is the method of conjugate gradients (also referred to as CG). It was first proposed in 1952 by Hestenes and Stiefel \cite{Hestenes-Stiefer-ConjGrad-1952} and since then, together with its related derivatives (e.g., conjugate gradient method on the normal equations (CGNE), least-square QR method (LSQR), etc.), it has been widely studied in the finite-dimensional setting (see the monographs \cite{Saad-2003_IterativeMethods,Simoncini-Szyld-2007,Liesen-Strakos-2003}) and also, though to a lesser extent, in the infinite-dimensional Hilbert space setting with bounded operators.

In order to describe the algorithm explicitly, let us introduce the solution manifold
\begin{equation}
 \mathcal{S}\;:=\;\{f\in\mathcal{D}(A)\,|\,Af=g\}
\end{equation}
relative to the problem \eqref{eq:invLP}. By assumption $\mathcal{S}$ is a convex, non-empty set in $\cH$ which is also closed, owing to the fact that $A$, being self-adjoint, is in particular a closed operator. (In fact, $\mathcal{S}$ is an affine space, owing to the linearity of $A$.) As a consequence, the projection map $P_\mathcal{S}:\cH\to\mathcal{S}$ is unambiguously defined and produces, for generic $x\in\cH$, the closest-to-$x$ point in $\mathcal{S}$. Observe that $P_\mathcal{S}$ is not a linear map.

In its iterative implementation, the conjugate gradient algorithm starts with an initial guess $f^{[0]}\in\cH$ and produces iterates $f^{[N]}$ according to a prescription that can be described in various equivalent ways \cite{Saad-2003_IterativeMethods,Liesen-Strakos-2003}, the most convenient of which for our purposes is
\begin{equation}\label{eq:CG-theta1}
 f^{[N]}\;:=\qquad\;\arg\!\!\!\!\!\!\!\!\!\!\!\!\min_{\!\!\!\!\!\!\!\!\!\!\!\!\! h\in\{f^{[0]}\}+\mathcal{K}_N(A,\mathfrak{R}_0)}\|A^{1/2}(h-P_\mathcal{S}f^{[0]})\|\,,\qquad N\in\mathbb{N}\,.
\end{equation}
More generally, we shall discuss conjugate gradient style algorithms with iterates given by
\begin{equation}\label{eq:CG-theta_generic}
 f^{[N]}\;:=\qquad\;\arg\!\!\!\!\!\!\!\!\!\!\!\!\min_{\!\!\!\!\!\!\!\!\!\!\!\!\! h\in\{f^{[0]}\}+\mathcal{K}_N(A,\mathfrak{R}_0)}\|A^{\theta/2}(h-P_\mathcal{S}f^{[0]})\|\,,\qquad N\in\mathbb{N}
\end{equation}
for some parameter $\theta\geqslant 0$ (the case $\theta=1$ being \emph{the} conjugate gradient method). It will be then convenient to refer to such $f^{[N]}$'s as the $\theta$-iterates. 

In \eqref{eq:CG-theta1}-\eqref{eq:CG-theta_generic} the vector $\mathfrak{R}_0$ is the zero-th order of the \emph{residuals} defined by
\begin{equation}
 \mathfrak{R}_N\;:=\;A f^{[N]}-g\,,\qquad N\in\mathbb{N}_0
\end{equation}
in terms of each iterate, and the vector space
\begin{equation}\label{eq:KN}
 \mathcal{K}_N(A,\mathfrak{R}_0)\;:=\;\mathrm{span}\{\mathfrak{R}_0,A\mathfrak{R}_0,\dots A^{N-1}\mathfrak{R}_0\}\,,\qquad N\in\mathbb{N}
\end{equation}
is the \emph{$N$-th order Krylov subspace} associated to $A$ and $\mathfrak{R}_0$.

Let us underline that \eqref{eq:CG-theta1}-\eqref{eq:CG-theta_generic} give the \emph{variational} characterisation of the conjugate gradient algorithm, and we start from such formulas because our subsequent study of the algorithm's convergence will be variational in nature; as well known, in practice the algorithm is implemented numerically through equivalent \emph{algebraic} versions \cite{Saad-2003_IterativeMethods,Simoncini-Szyld-2007,Liesen-Strakos-2003}, that produce the same iterates $f^{[N]}$ without of course requiring the a priori knowledge of the solution $P_\mathcal{S}f^{[0]}$.

Clearly, the above definitions are all well-posed if $A$ is bounded, whereas for \eqref{eq:KN} and hence \eqref{eq:CG-theta1}-\eqref{eq:CG-theta_generic} to make sense for any $N$ when $A$ is \emph{unbounded}, additional technical assumptions are needed in order to avoid possible domain issues. We shall discuss them in the general set-up of the problem presented in Section \ref{sec:setup_mainresults} -- but let us emphasise already at this stage that even when domain issues are taken under control, the unbounded-case framework that we are considering in this work remains a non-trivial generalisation at all of the bounded case.

As well known, for \emph{finite-dimensional} inverse problems CG is an extremely popular, versatile, and efficient numerical scheme -- it belongs, in particular, to the class of Krylov subspace methods, that are sometimes  even counted among the `Top 10 Algorithms' of the 20th century \cite{Dongarra-Sullivan-Best10-2000,Cipra-SIAM-News} -- and the convergence of $f^{[N]}$ to the exact solution $f$ is by now a classical and deeply understood theory (see, e.g., the monographs \cite{Saad-2003_IterativeMethods,Liesen-Strakos-2003}).

The convergence theory of CG has been markedly less explored in the setting of \emph{infinite-dimensional} $\cH$, a line of investigation in which yet important works have been produced over the last five decades, both in the scenario where $A$ is bounded with everywhere-defined bounded inverse \cite{Daniel-1967,Daniel-1970,Herzog-Ekkehard-2015}, or at least with bounded inverse on its range \cite{Kammerer-Nashed-1972}, and in the scenario where $A$ is bounded with possible unbounded inverse on its range \cite{Kammerer-Nashed-1972,Nemirovskiy-Polyak-1985,Nemirovskiy-Polyak-1985-II,Louis-1987,Brakhage-1987}.

In contrast, the scenario where $A$ is \emph{unbounded} has been only recently considered from special perspectives, in particular in view of existence \cite{Olver-2009} (for GMRES algorithms), or convergence when $A$ is regularised and made invertible with everywhere-defined bounded inverse \cite{Gilles-Townsend-2019}, whereas the general convergence theory (that is, including the case where \eqref{eq:invLP} is ill-posed) is virtually unexplored.

(We should also like to mention the ongoing related analysis on abstract Krylov methods in infinite dimension and with possible unbounded $A$: see \cite{CMN-2018_Krylov-solvability-bdd,CM-2019_ubddKrylov,CM-KrylovPerturbations-2020,CMN-truncation-2018} and the references therein.)

In the present work we establish a class of convergence results for the conjugate gradient algorithm precisely in the most general setting where $A$ is unbounded and the associated inverse problem \eqref{eq:invLP} is possibly ill-posed. This applies, in particular, to the ubiquitous case where $A=-\Delta$ on $L^2(\mathbb{R}^d)$. Our analysis consists of a non-trivial generalisation of the very subtle approach by Nemirovskiy and Polyak from their above-mentioned 1984 work \cite{Nemirovskiy-Polyak-1985} for bounded $A$.

In such work the full convergence estimates of the error $\|f^{[N]}-f\|$ and residual $\|Af^{[N]}-g\|$ were proved, and in the follow-up work \cite{Nemirovskiy-Polyak-1985-II} the results were shown to be \emph{optimal} in the sense that for the entire class of bounded, ill-posed problems, one can do no better than the estimates provided. The boundedness of $A$ was crucial in a two-fold way. First, it forced the blow-up of a suitable sum ($\delta_N$, in their notation -- see \eqref{eq:defdeltaN} below) of the reciprocals of the $N$ zeroes of a polynomial that represents the minimisation \eqref{eq:CG-theta_generic}: since no such zero can exceed $\|A\|$, the reciprocals cannot vanish and their sum necessarily diverges. As a consequence, the error and the residual, which in turn can be controlled by an inverse power of $\delta_N$, are then shown to vanish as $N\to\infty$, thus establishing convergence. Second, boundedness of $A$ was also determinant to quantify the convergence, as the latter was boiled down to a min-max procedure for polynomials on the \emph{finite} spectral interval containing $\sigma(A)$, then on such an interval (suitable modifications of) the Chebyshev polynomials are recognised to optimise the rate of convergence, and explicit properties of (the zeroes of) Chebyshev polynomials finally provide a quantitative version of the vanishing of error and residual.

In our approach we are able to bypass the restriction of the finiteness of $\|A\|$ as far as the convergence alone is concerned. As for the quantitative rate, the min-max strategy of \cite{Nemirovskiy-Polyak-1985} by no means can be adapted to polynomials over the whole $[0,+\infty)$ and in fact a careful analysis of the structure of the proof of \cite{Nemirovskiy-Polyak-1985,Nemirovskiy-Polyak-1985-II}, as we shall comment in due time, seems to indicate that if $A$ is unbounded with unbounded inverse on its range, then the convergence rate can be arbitrarily small.

\bigskip

The discussion is organised as follows. In Section \ref{sec:setup_mainresults} we introduce the rigorous set-up of the convergence problem for the conjugate gradient method and we state and comment our main result. In Section \ref{sec:technical} we develop an amount of preparatory materials of algebraic and measure-theoretic nature, which are needed to finally prove our main theorem in Section \ref{sec:mainproof}. After the proof, an amount of remarks are collected with the purpose of clarifying the importance of certain technical steps and, above all (Remarks \ref{rem:comparison1}-\ref{rem:comparison2}) the actual novelties and differences of the present scheme as compared to the bounded case scenario. Last, in Section \ref{sec:num_test} we discuss a selection of numerical tests that confirm the main features of our convergence result and corroborate our intuition on certain relevant differences with respect to the bounded case.

\section{ Set-up and main results}\label{sec:setup_mainresults}

Let us start with the rigorous formulation of all the notions needed for our convergence result. Here and in the following $A$ is a non-negative, densely defined, self-adjoint operator on a Hilbert space $\cH$, including the possibility that $A$ be unbounded and with a non-trivial kernel.

First, one needs to ensure that the conjugate gradient iterates are \emph{well-defined}. As mentioned in the Introduction, one chooses a datum $g\in\ran A$ and an initial guess $f^{[0]}\in\cH$, and for some $\theta\geqslant 0$ defines the $\theta$-iterates 
\begin{equation}\label{eq:thetaiterates}
 f^{[N]}\;:=\qquad\;\arg\!\!\!\!\!\!\!\!\!\!\!\!\min_{\!\!\!\!\!\!\!\!\!\!\!\!\! h\in\{f^{[0]}\}+\mathcal{K}_N(A,\mathfrak{R}_0)}\|A^{\theta/2}(h-P_\mathcal{S}h)\|\,,\qquad N\in\mathbb{N}
\end{equation}
with
\begin{eqnarray}
 \mathfrak{R}_N\!\!&=&\!\!A f^{[N]}-g\,,\qquad N\in\mathbb{N}_0\,, \\
 \mathcal{K}_N(A,\mathfrak{R}_0)\!\!&=&\!\!\mathrm{span}\{\mathfrak{R}_0,A\mathfrak{R}_0,\dots A^{N-1}\mathfrak{R}_0\}\,,\qquad N\in\mathbb{N}\,. \label{eq:KNAiteration}
\end{eqnarray}

In order to apply an arbitrary positive power of $A$ to $Af^{[0]}-g$, we require that both $g$ and $f^{[0]}$ be $A$-smooth vectors \cite[Sect.~X.6]{rs2}, meaning that they belong to the space
\begin{equation}\label{eq:gsmooth}
 C^\infty(A)\;:=\;\bigcap_{N\in\mathbb{N}}\mathcal{D}(A^N)\,.
\end{equation}
In the applications where $A$ is a differential operator, $A$-smoothness is a regularity requirement.

In turn, $A$-smoothness of $g$ and $f^{[0]}$ implies $\mathcal{K}_N(A,\mathfrak{R}_0)\subset C^\infty(A)$, and obviously $P_\mathcal{S}h\in\mathcal{S}\subset C^\infty(A)$, whereas by interpolation $C^\infty(A)\subset\mathcal{D}(A^{\theta/2})$ for any $\theta\geqslant 0$. This guarantees that in the minimisation \eqref{eq:thetaiterates} one is allowed to apply $A^{\theta/2}$ to any vector $h-P_\mathcal{S}h$.

We have thus seen that under the assumptions
\begin{equation}\label{eq:assumptionforthetaiterates}
 g\;\in\;\ran A\cap C^\infty(A)\,,\qquad f^{[0]}\;\in\;C^\infty(A) 
\end{equation}
the corresponding $\theta$-iterates $f^{[N]}$ are unambiguously defined by \eqref{eq:thetaiterates}-\eqref{eq:KNAiteration} above for \emph{any} $\theta\geqslant 0$. If $A$ is bounded, \eqref{eq:assumptionforthetaiterates} simply reduces to $g\in\ran A$.
In fact, \eqref{eq:assumptionforthetaiterates} are \emph{minimal assumptions}, inescapable if one wants to give meaning to conjugate gradient iterates in the unbounded case.

Such iterates have three notable properties, whose proof is deferred to Section \ref{sec:technical}.

\begin{proposition}\label{prop:FNproperties}
 The $\theta$-iterates $f^{[N]}$ defined for a given $\theta\geqslant 0$ by means of \eqref{eq:thetaiterates}-\eqref{eq:KNAiteration} under the assumption \eqref{eq:assumptionforthetaiterates} satisfy
 \begin{eqnarray}
  f^{[N]}-P_\mathcal{S}f^{[N]}\!\!&\in&\!\!(\ker A)^\perp\qquad\qquad\quad\qquad\, \forall N\in\mathbb{N}_0\,, \label{eq:fNPSFN-kerperp}\\
  P_\mathcal{S} f^{[N]}\!\!&=&\!\!P_\mathcal{S} f^{[0]}\qquad\qquad\qquad\qquad\; \forall N\in\mathbb{N}\,, \label{eq:PFNPF0} \\
  f^{[N]}-P_\mathcal{S}f^{[N]}\!\!&=&\!\! p_N(A)(f^{[0]}-P_\mathcal{S}f^{[0]}) \qquad \forall N\in\mathbb{N}\,, \label{eq:iterativeNemi}
 \end{eqnarray}
where $p_N(\lambda)$ is for each $N$ a polynomial of degree at most $N$ and such that $p_N(0)=1$.
%
%
\end{proposition}

As, by \eqref{eq:PFNPF0}, all such $f^{[N]}$'s have the same projection onto the solution manifold $\mathcal{S}$, the approach of $f^{[N]}$ to $\mathcal{S}$ consists explicitly of a convergence $f^{[N]}\to P_\mathcal{S} f^{[0]}$. Let us now specify in which sense this convergence is to be monitored.

The underlying idea, as is clear in the typical applications where $A$ is a differential operator on a $L^2$-space, is that $\|f^{[N]}- P_\mathcal{S} f^{[0]}\|_{(A)}\to 0$ in some $A$-dependent Sobolev norm. For this to make sense, clearly one needs enough $A$-regularity on $f^{[N]}- P_\mathcal{S} f^{[0]}$, which eventually is guaranteed by the regularity initially assumed on $f^{[0]}$. Thus, the general indicator of convergence has the form $\|A^{\sigma/2}(f^{[N]}- P_\mathcal{S} f^{[0]})\|$, but an extra care is needed if one wants to control the convergence in the abstract analogue of a low-regularity, negative-Sobolev norm, which would amount to formally consider $\sigma<0$, for in general $A$ can have a kernel and hence is only invertible on its range.

Based on these considerations, and inspired by the analogous discussion in \cite{Nemirovskiy-Polyak-1985} for bounded $A$, let us introduce the class $\mathscr{C}_{A,g}(\theta)$ of vectors of $\cH$ defined for generic $\theta\in\mathbb{R}$ as
\begin{equation}\label{eq:defCtheta}
 \mathscr{C}_{A,g}(\theta)\;:=\;
 \begin{cases}
  \{x\in\cH\,|\,x-P_\mathcal{S}x\in\mathcal{D}(A^{\frac{\theta}{2}})\}\,, & \theta\geqslant 0\,, \\
  \{x\in\cH\,|\,x-P_\mathcal{S}x\in\ran (A^{-\frac{\theta}{2}})\}\,, & \theta <0\,.
 \end{cases}
\end{equation}
(The dependence of $\mathscr{C}_{A,g}(\theta)$ on $g$ is implicit through the solution manifold $\mathcal{S}$.)
Distinguishing the two cases in \eqref{eq:defCtheta} is needed whenever $A$ has a non-trivial kernel. If instead $A$ is injective, and so too is therefore $A^{-\frac{\theta}{2}}$ for $\theta<0$, then $A^{-\frac{\theta}{2}}$ is a bijection between the two dense subspaces $\mathcal{D}(A^{-\frac{\theta}{2}})=\ran(A^{\frac{\theta}{2}})$ and $\ran(A^{-\frac{\theta}{2}})=\mathcal{D}(A^{\frac{\theta}{2}})$ of $\cH$.

Related to the class $\mathscr{C}_{A,g}(\theta)$ we have two further useful notions. One, for fixed $\theta\in\mathbb{R}$ and $x\in \mathscr{C}_{A,g}(\theta)$, is the vector
\begin{equation}\label{eq:uthetadef}
 u_\theta(x)\;:=\;
 \begin{cases}
  A^{\frac{\theta}{2}}(x-P_\mathcal{S}x)\,, & \theta\geqslant 0\,, \\
  \textrm{the minimal norm solution $u$ to}\; A^{-\frac{\theta}{2}}u=x-P_\mathcal{S}x\,, & \theta <0\,.
 \end{cases}
\end{equation}
The other is the functional $\rho_\theta$ defined on the vectors $x\in \mathscr{C}_{A,g}(\theta)$ as
\begin{equation}\label{eq:defrhothetawithu}
 \rho_\theta(x)\;:=\;\|u_\theta(x)\|^2\,.
\end{equation}
Thus,
\begin{equation}\label{eq:rhothetaexiplicit}
 \rho_\theta(x)\;=\;
 \begin{cases} 
  \quad\;\|A^{\frac{\theta}{2}}(x-P_\mathcal{S}x)\|^2\,, &  \qquad\theta\geqslant 0\,, \\
 \;\Big\|\Big(A^{-\frac{\theta}{2}}\Big|_{\ran\big(A^{-\frac{\theta}{2}}\big)}\Big)^{\!-1}(x-P_\mathcal{S}x)\Big\|^2\,, & \qquad\theta <0\,,
 \end{cases}
\end{equation}
with an innocent abuse of notation in \eqref{eq:rhothetaexiplicit} when $\theta<0$, as the operator inverse is to be understood for a (self-adjoint, and positive-definite) operator on the Hilbert subspace $\overline{\ran A}$.

It is worth remarking that in the special case when $A$ is \emph{bounded}, the following interesting properties hold, whose proof is deferred to Section \ref{sec:technical}, which do not have a counterpart in the unbounded case except for the obvious identity $\mathscr{C}_{A,g}(0)=\cH$.

\begin{lemma}\label{lem:Abdd_Etheta}
 If $A$ (besides being self-adjoint and non-negative) is bounded, and if $g\in \ran A$, then:
 \begin{itemize}
  \item[(i)] $\mathscr{C}_{A,g}(\theta)=\cH$ whenever $\theta\geqslant 0 $;
  \item[(ii)] $\mathscr{C}_{A,g}(\theta)\subset \mathscr{C}_{A,g}(\theta') $ for $\theta\leqslant \theta'$;
  \item[(iii)] for $\theta\leqslant \theta'$ and $x\in \mathscr{C}_{A,g}(\theta)$ one has $u_{\theta'}(x)=A^{(\theta'-\theta)/2}u_{\theta}(x)$, whence also $\rho_{\theta'}(x)\:\leqslant\;\|A\|^{\theta'-\theta}\rho_{\theta}(x)$.
 \end{itemize}
\end{lemma}

Back to the general case where $A$ is possibly \emph{unbounded}, the goal is to evaluate certain $\rho_\sigma$-functionals along the sequence of the $f^{[N]}$'s. This may require an extra assumption on the initial guess $f^{[0]}$, as the following Lemma shows.

\begin{lemma}\label{lem:sigma-assumptions}
 Consider the $\theta$-iterates $f^{[N]}$ defined for a given $\theta\geqslant 0$ by means of \eqref{eq:thetaiterates}-\eqref{eq:KNAiteration} under the assumption \eqref{eq:assumptionforthetaiterates}. Then:
 \begin{itemize}
  \item[(i)] $f^{[N]}\in \mathscr{C}_{A,g}(\sigma)$ $\forall \sigma\geqslant 0$;
  \item[(ii)]  $f^{[N]}\in \mathscr{C}_{A,g}(\sigma)$ for any $\sigma<0$ such that, additionally, $f^{[0]}\in \mathscr{C}_{A,g}(\sigma)$, in which case
  \begin{equation}\label{eq:usigma-uzero}
   u_\sigma(f^{[N]})\;=\;p_N(A)\,u_\sigma(f^{[0]})\,,
  \end{equation}
  where $p_N(\lambda)$ is precisely the polynomial mentioned in Proposition \ref{prop:FNproperties}.
 \end{itemize}
\end{lemma}

Lemma \ref{lem:sigma-assumptions} is a direct consequence of \eqref{eq:iterativeNemi} in Proposition \ref{prop:FNproperties} above: for completeness we include its proof in Section \ref{sec:technical}.

It then makes sense to control the convergence $f^{[N]}\to \mathcal{P}_Sf^{[0]}$ in the $\rho_\sigma$-sense, for $\sigma$ positive or negative, with suitable assumptions on $f^{[0]}$. Explicitly,
\begin{equation}\label{eq:rhoindicator}
 \begin{split}
 \rho_\sigma(f^{[N]})\;&=\;\|u_\sigma(f^{[N]})\|^2 \\
 &=\;
 \begin{cases} 
  \quad\;\|A^{\frac{\sigma}{2}}(f^{[N]}-P_\mathcal{S}f^{[0]})\|^2\,, &  \qquad\sigma\geqslant 0\,, \\
 \;\Big\|\Big(A^{-\frac{\sigma}{2}}\Big|_{\ran\big(A^{-\frac{\sigma}{2}}\big)}\Big)^{\!-1}(f^{[N]}-P_\mathcal{S}f^{[0]})\Big\|^2\,, & \qquad\sigma <0\,,
 \end{cases}
 \end{split}
\end{equation}
having used \eqref{eq:PFNPF0}.
%
The most typical and meaningful choices in the applications are
\begin{equation}\label{eq:xi123}
 \begin{split}
  \rho_0(f^{[N]})\;&=\;\big\|f^{[N]}-\mathcal{P}_Sf^{[0]}\big\|^2\,, \\
  \rho_1(f^{[N]})\;&=\;\big\langle f^{[N]}-\mathcal{P}_Sf^{[0]},A\,(f^{[N]}-\mathcal{P}_Sf^{[0]})\big\rangle\,, \\
  \rho_2(f^{[N]})\;&=\;\big\|A(f^{[N]}-\mathcal{P}_Sf^{[0]})\big\|^2\,,
 \end{split}
\end{equation}
that is, respectively, the norm of the error, the so-called `energy' (semi-)norm, and the norm of the residual.


The preparation made so far for our forthcoming main result (Theorem \ref{thm:main} below) does not account yet for the necessity of one further, restrictive assumption on the datum $g$ and the initial guess $f^{[0]}$ of the algorithm, a restriction needed once again to deal with the possible unboundedness of the operator $A$ (the bounded case being controllable for arbitrary $g\in\mathrm{ran} A$ and $f^{[0]}\in\cH$). The actual need of a special, inevitable choice of $g$ and $f^{[0]}$ will be fully clear in the course of the proof; for the time being, let us outline here a short heuristic reasoning.

From the expression of the indicator of convergence $\rho_\sigma(f^{[N]})$, for concreteness the case $\sigma\geqslant 0$ in \eqref{eq:rhoindicator}, and from the iterates properties \eqref{eq:PFNPF0}-\eqref{eq:iterativeNemi} announced in Proposition \ref{prop:FNproperties}, it is easy to realise, as we shall argue in the next Section, that the actual quantity to control along the limit $N\to\infty$ is an integral of the form
\[
 \int_{[0,+\infty)}\big|\lambda^{\frac{\sigma}{2}}p_N(\lambda)\big|^2\,\ud\langle f^{[0]}-P_{\mathcal{S}}f^{[0]},E^A(\lambda)( f^{[0]}-P_{\mathcal{S}}f^{[0]})\rangle\,,
\]
for suitable polynomials $p_N$ determined by the minimisation \eqref{eq:thetaiterates}, where the measure is the scalar spectral measure associated to the self-adjoint operator $A$, and hence (by positivity of $A$) the integration runs over $[0,+\infty)$. For bounded $A$'s the integration is actually restricted within the spectrum of $A$, hence within a compact interval of the non-negative half line, and this naturally provides a kind of \emph{uniformity} in $N$ that is crucial in controlling the vanishing of the above integral in the limit. When instead $A$ (and hence the integration domain) is unbounded, some other source of uniformity in $N$ must be implemented, which eventually is to be some kind of uniformity of the measure $\lambda^{2N}\ud\langle f^{[0]}-P_{\mathcal{S}}f^{[0]},E^A(\lambda)( f^{[0]}-P_{\mathcal{S}}f^{[0]})\rangle$, in other words, a suitable \emph{control of the growth in $N$ of the norm} $\|A^N(f^{[0]}-P_{\mathcal{S}}f^{[0]})\|$. In turn, this requires a control in $N$ of $\|A^N f^{[0]}\|$ and $\|A^N g\|$.

It is with the above heuristics in mind that we recall the following classes of vectors \cite[Definition 7.1]{schmu_unbdd_sa}: the \emph{analytic} vectors for $A$ are the elements of the \emph{subspace}
\begin{equation}\label{def:analytic-vectors}
 \mathcal{D}^a(A)\;:=\;\left\{ g\in C^\infty(A)=\bigcap_{n\in\mathbb{N}}\mathcal{D}(A^n)\left|\!
 \begin{array}{c}
  \|A^ng\|\leqslant C^n_g n! \\
  \textrm{ for any $n\in\mathbb{N}$} \\
  \textrm{and some $C_g>0$}
 \end{array}
 \!\!\right.\right\},
\end{equation}
and the \emph{quasi-analytic} vectors for $A$ are the elements of the \emph{set}
 \begin{equation}\label{def:quasianalytic-vectors}
   \mathcal{D}^{qa}(A)\;:=\;\Big\{ g\in C^\infty(A)\,\Big|\,\sum_{n\in\mathbb{N}}\|A^n g\|^{-\frac{1}{n}}=+\infty\Big\}\,.
 \end{equation}
Clearly $\mathcal{D}^a(A)\subset\mathcal{D}^{qa}(A)\subset C^\infty(A)$, and the self-adjointness of $A$ ensures that the subspace of its analytic vectors is \emph{dense} in $\cH$ (this is the celebrated Nelson theorem: see, e.g., \cite[Theorem 7.16]{schmu_unbdd_sa}). Obviously when $A$ is bounded the whole $\cH$ is made of analytic vectors for $A$.

We are finally in the condition to formulate our main result.

\begin{theorem}\label{thm:main}
Let $A$ be a non-negative self-adjoint operator on the Hilbert space $\cH$. Let
 \begin{equation}\label{eq:maincond1}
 g\;\in\;\mathcal{D}^a(A)\cap \ran A\,.
\end{equation}
Consider the conjugate gradient algorithm associated with $A$ and $g$  where the initial guess vector $f^{[0]}$ satisfies
\begin{equation}\label{eq:maincond2}
 f^{[0]}\;\in\;\mathcal{D}^a(A)\cap \mathscr{C}_{A,g}(\sigma^*)\,,\qquad \sigma^*=\min\{\sigma,0\}
\end{equation}
for a given $\sigma\in\mathbb{R}$, and where the iterates $f^{[N]}$, $N\in\mathbb{N}$, are constructed via \eqref{eq:thetaiterates} with parameter $\theta=\xi\geqslant 0$ under the condition $\sigma\leqslant\xi$. Then
\begin{equation}\label{eq:main-cg-convergence}
 \lim_{N\to\infty}\rho_\sigma(f^{[N]})\;=\;0\,.
\end{equation}
\end{theorem}

As a corollary of the proof that we shall discuss, we also have:

\begin{corollary}\label{cor:main}
  The same conclusion of Theorem \ref{thm:main} follows also when the assumptions on $g$ and $f^{[0]}$ are replaced by
\begin{equation}\label{eq:maincond3}
 \begin{split}
   g\;&\in\;C^\infty(A)\cap \ran A\,, \\
   f^{[0]}\;&\in\;C^\infty(A)\cap \mathscr{C}_{A,g}(\sigma^*)\,, \\
   f^{[0]}-P_\mathcal{S}f^{[0]}&\in\;\mathcal{D}^{qa}(A)\,,
 \end{split}
\end{equation}
or by
\begin{equation}\label{eq:maincond4}
 \begin{split}
   g\;&\in\;\mathcal{D}^{qa}(A)\cap \ran A\,, \\
   f^{[0]}\;&=\;0\,.
 \end{split}
\end{equation}
\end{corollary}

In other words, Theorem \ref{thm:main} states that the convergence holds at a given `$A$-regularity level' $\sigma$ for $\xi$-iterates built with \emph{equal or higher} `$A$-regularity level' $\xi\geqslant\sigma$, and with an initial guess $f^{[0]}$ that is $A$-analytic if $\sigma\geqslant 0$, and additionally belongs to the class $\mathscr{C}_{A,g}(\sigma)$ if $\sigma<0$.

In particular, with no extra assumption on $f^{[0]}$ but its $A$-analyticity, the $\xi$-iterates with $\xi\geqslant 0$ automatically converge in the sense of the error ($\sigma=0$, see \eqref{eq:xi123} above), the $\xi$-iterates with $\xi\geqslant 1$ automatically converge in the sense of the error and of the energy norm ($\sigma=1$), the $\xi$-iterates with $\xi\geqslant 2$ automatically converge in the sense of the error, energy norm, and residual ($\sigma=2$).

\begin{remark}\label{rem:finite_steps}
 If, for a finite $N$, $\rho_\sigma(f^{[N]})=0$, then the very iterate $f^{[N]}$ \emph{is} a solution to the linear problem $Af=g$, and one says that the algorithm `has come to convergence' in a finite number ($N$) of steps. Indeed, $\rho_\sigma(f^{[N]})=0$ is the same as $A^{\frac{\sigma}{2}}(f^{[N]}-P_\mathcal{S}f^{[0]})=0$ if $\sigma\geqslant 0$, i.e., $f^{[N]}-P_\mathcal{S}f^{[0]}\in\ker A^{\frac{\sigma}{2}}=\ker A$; this, combined with $f^{[N]}-P_\mathcal{S}f^{[0]}\in(\ker A)^\perp$ (see \eqref{eq:fNPSFN-kerperp}-\eqref{eq:PFNPF0} above), implies that $f^{[N]}=P_\mathcal{S}f^{[0]}\in\mathcal{S}$. On the other hand, $\rho_\sigma(f^{[N]})=0$ is the same as $u_\sigma(f^{[N]})=0$ with $A^{-\frac{\sigma}{2}}u_\sigma(f^{[N]})=f^{[N]}-P_\mathcal{S}f^{[0]}$ if $\sigma < 0$, whence again $f^{[N]}=P_\mathcal{S}f^{[0]}\in\mathcal{S}$. 
\end{remark}

\begin{remark}\label{rem:Abdd}~
\begin{itemize}
 \item[(i)] In the special scenario where $A$ is (everywhere-defined and) bounded, $A$-analyticity is automatically guaranteed, so one only needs to assume that $g\in\ran A$ and $f^{[0]}\in \mathscr{C}_{A,g}(\sigma^*)$ for some $\sigma\in\mathbb{R}$ ($\sigma^*=\min\{\sigma,0\}$) in order for the convergence of the $\xi$-iterates ($\xi\geqslant\sigma$) to hold in the sense $\rho_\sigma(f^{[N]})\to 0$. Then, owing to Lemma \ref{lem:Abdd_Etheta}, one automatically has also $\rho_{\sigma'}(f^{[N]})\to 0$ for any $\sigma'\geqslant \sigma$. This is precisely the form of the convergence result originally established by Nemirovskiy and Polyak \cite{Nemirovskiy-Polyak-1985}. 
 \item[(ii)] Thus, in the bounded-case scenario, if $\sigma$ is the minimum level of convergence chosen, then not only are the $\xi$-iterates with $\xi\geqslant\sigma$ proved to $\rho_\sigma$-converge, but in addition the \emph{same} $\xi$-iterates also $\rho_{\sigma'}$-converge at any other level $\sigma'\geqslant \sigma$, with no upper bound on $\sigma'$. 
 In particular, it is shown in \cite{Nemirovskiy-Polyak-1985} that
 \begin{equation}\label{eq:Nem-conv-rate}
  \rho_{\sigma'}(f^{[N]})\;\leqslant\;C(\|A\|_{\mathrm{op}},\xi-\sigma)\,(2N+1)^{-2(\sigma'-\sigma)}\,\rho_{\sigma}(f^{[0]})\,,\quad \sigma<\sigma'\leqslant\xi\,,
 \end{equation}
 for some constant $C(\|A\|_{\mathrm{op}},\xi-\sigma)>0$, thus providing an explicit \emph{rate of convergence} of the $\xi$-iterates in a generic $\rho_{\sigma'}$-sense such that $\sigma'\in(\sigma,\xi]$.
 
 \item[(iii)] In the general unbounded-case scenario, instead, the $\rho_\sigma$-convergence guaranteed by Theorem \ref{thm:main} is not exportable to $\rho_{\sigma'}$-convergence with $\sigma'>\sigma$.
\end{itemize}
\end{remark}

\begin{remark}
 When, in the unbounded case, $A$ has an everywhere-defined bounded inverse, one has $\mathscr{C}_{A,g}(\sigma)=\cH$ for any $\sigma\leqslant 0$. Therefore, Theorem \ref{thm:main} guarantees the $\rho_\sigma$-convergence of the $\xi$-iterates for any $\sigma\leqslant 0$, provided that $g$ and $f^{[0]}$ are $A$-analytic. Such `weaker' convergence can be still informative in many contexts. For instance, choosing
\[
 \begin{split}
  \cH&=L^2(\mathbb{R}^d) \\
  A&=-\Delta+\mathbbm{1}\quad\textrm{with}\quad \mathcal{D}(A)=H^2(\mathbb{R}^d)\qquad \textrm{($\ran A=\cH$)} \\
  g,f^{[0]}&\in C^\infty(\mathbb{R}^d)\,,
 \end{split}
\]
we see that the $\theta$-iterates defined by \eqref{eq:CG-theta_generic} with the above data converge to the unique solution $f$ to the inverse problem $-\Delta f +f =g$ in any negative Sobolev space $H^{\sigma}(\mathbb{R}^d)$, $\sigma < 0$; in particular, $f^{[N]}(x)\to f(x)$ point-wise almost everywhere.
\end{remark}

\begin{remark}\label{rem:analytic-quasian}~
\begin{itemize}
 \item[(i)] Assumptions \eqref{eq:maincond1}-\eqref{eq:maincond2} of Theorem \ref{thm:main}, as well as assumptions \eqref{eq:maincond3} of Corollary \ref{cor:main}, are needed to cover the case of our primary interest, the \emph{unboundedness} of $A$. 
 \item[(ii)] Such restrictions still allow the admissible $g$ and $f^{[0]}$ to run over a \emph{dense} of $\cH$. 
 \item[(iii)] Assumptions \eqref{eq:maincond3} are slightly less restrictive than \eqref{eq:maincond1}-\eqref{eq:maincond2}. Indeed, from \eqref{eq:maincond1}-\eqref{eq:maincond2}, since $g$ is analytic, so is $P_{\mathcal{S}}f^{[0]}$ (a fact that we shall prove in Lemma \ref{lem:fpPf-analitic}), and by linearity $f^{[0]}-P_\mathcal{S}f^{[0]}$ is analytic too, whence \eqref{eq:maincond3}. 
 \item[(iv)] Albeit more general, assumptions \eqref{eq:maincond3} have the apparent drawback of being formulated in terms of a vector, $P_\mathcal{S}f^{[0]}$, that is unknown prior to actually solving the inverse problem.  We singled out \eqref{eq:maincond3} because, as is going to emerge in the forthcoming discussion, it is precisely the quasi-analyticity of $f^{[0]}-P_\mathcal{S}f^{[0]}$ (together with the inevitable operational assumption $g\in C^\infty(A)\cap \ran A$) that makes our proof work. In fact, quasi-analyticity of $f^{[0]}-P_\mathcal{S}f^{[0]}$ provides a control on the $N$-growth of $\|A^N(f^{[0]}-P_{\mathcal{S}}f^{[0]})\|$, the quantity we heuristically discussed prior to stating Theorem \ref{thm:main}. 
 \item[(v)] Assumptions \eqref{eq:maincond4} are in fact a special case of \eqref{eq:maincond3}, as will be clear from Lemma \ref{lem:fpPf-analitic}(i). We singled them out to connect our result with the frequent occurrence, in conjugate gradient methods, where the initial guess $f^{[0]}$ is just the zero function.
\end{itemize}
\end{remark}

\begin{remark}\label{rem:generic-regime}
 As a follow-up of Remark \ref{rem:analytic-quasian}(iii): although the quasi-analyticity of $f^{[0]}-P_\mathcal{S}f^{[0]}$ is only a \emph{sufficient} condition, some possibly weaker assumption of that sort, namely some kind of control of the growth in $N$ of $\|A^N(f^{[0]}-P_{\mathcal{S}}f^{[0]})\|$, is surely needed for the conjugate gradient convergence \eqref{eq:main-cg-convergence}. That in the regime $g,f^{[0]}\in C^\infty(A)$ the vanishing $\rho_\sigma(f^{[N]})\to 0$ is \emph{not} guaranteed, is going to be explained in Proposition \ref{prop:generic-regime}, when the technical details of Theorem \ref{thm:main} will be clear.  
\end{remark}

\section{Intermediate technical facts}\label{sec:technical}

We discuss in this Section an amount of technical properties that are needed for the proof of the main Theorem \ref{thm:main}.

For convenience, let us set for each $N\in\mathbb{N}$
\begin{equation}
 \begin{split}
  \mathbb{P}([0,+\infty))\;&:=\;\{\textrm{real-valued polynomials $p(\lambda)$, $\lambda\in[0,+\infty)$} \}  \\
  \mathbb{P}_N\;&:=\;\{\;p\in \mathbb{P}([0,+\infty))\,|\,\deg p\leqslant N \, \} \\
  \mathbb{P}_N^{(1)}\;&:=\;\{\;p\in \mathbb{P}_N\,|\,p(0)=1\,\}\,.
 \end{split}
\end{equation}

Let us start with the proof of those properties stated in Section \ref{sec:setup_mainresults}. The proof of Proposition \ref{prop:FNproperties} requires the following elementary fact.

\begin{lemma}\label{lem:orthlemma}
 Let $z\in\cH$. For a point $y\in\mathcal{S}$ these conditions are equivalent:
 \begin{itemize}
  \item[(i)] $y=P_Sz$\,,
  \item[(ii)] $z-y\in(\ker A)^\perp$\,.
 \end{itemize}
\end{lemma}

\begin{proof}
 By linearity of $A$, $\mathcal{S}=\{y\}+\ker A$. If $z-y\in(\ker A)^\perp$, then for any $x\in\ker A$, and hence for a generic point $y+x\in\mathcal{S}$, one has 
 \[
  \|z-(y+x)\|^2\;=\;\|z-y\|^2+\|x\|^2\;\geqslant\;\|z-y\|^2\,,
 \]
therefore $y$ is necessarily the closest to $z$ among all points in $\mathcal{S}$, i.e., $y=P_Sz$. This proves that (ii) $\Rightarrow$ (i). Conversely, if $y=P_Sz$, and if by contradiction $z-y$ does not belong to $(\ker A)^\perp$, then $\langle x_0,z-y\rangle>0$ for some $x_0\in\ker A$. In this case, let us consider the polynomial
\[
 p(t)\;:=\;\|z-y-tx_0\|^2\;=\;\|x_0\|^2 t^2-2\, \langle x_0,z-y\rangle\,t+\|z-y\|^2\,.
\]
Clearly, $t=0$ is not a point of minimum for $p(t)$, as for $t>0$ and small enough one has $p(t)\leqslant p(0)$. This shows that there are points $y+tx_0\in\mathcal{S}$ for which $\|z-(y+tx_0)\|\leqslant \|z-y\|$, thus contradicting the assumption that $y$ is the closest to $z$ among all points in $\mathcal{S}$. Then necessarily $z-y\in(\ker A)^\perp$, which proves that (i) $\Rightarrow$ (ii).
\end{proof}

\begin{proof}[Proof of Proposition \ref{prop:FNproperties}]
  In the minimisation \eqref{eq:thetaiterates}
 \[
  h-f^{[0]}\;=\;q_{N-1}(A)(Af^{[0]}-g)\;=\;q_{N-1}(A)\,A\,(f^{[0]}-P_\mathcal{S}f^{[0]})
 \]
 for some polynomial $q_{N-1}\in \mathbb{P}_{N-1}$, whence also
 \[
  h-P_\mathcal{S}f^{[0]}\;=\;q_{N-1}(A)\,A\,(f^{[0]}-P_\mathcal{S}f^{[0]})+(f^{[0]}-P_\mathcal{S}f^{[0]})\,.
 \]
 This implies, upon setting $p_N(\lambda):=\lambda\, q_{N-1}(\lambda)+1$, that
  \begin{equation*}\tag{*}
  f^{[N]}-P_\mathcal{S}f^{[0]}\;=\;p_N(A)(f^{[0]}-P_\mathcal{S}f^{[0]}) \qquad \forall N\in\mathbb{N}\,,
 \end{equation*}
 where $p_N\in\mathbb{P}_N^{(1)}$.

 Moreover, $f^{[N]}-P_S f^{[N]}\in(\ker A)^\perp$, as a consequence of Lemma \ref{lem:orthlemma} applied to the choice $z=f^{[N]}$ and $y=P_S f^{[N]}$. With an analogous argument, also $f^{[0]}-P_S f^{[0]}\in(\ker A)^\perp$. Thus, \eqref{eq:fNPSFN-kerperp} is proved.

 Owing to \eqref{eq:assumptionforthetaiterates} and \eqref{eq:fNPSFN-kerperp}, $f^{[0]}-P_S f^{[0]}\in (\ker A)^\perp\cap C^\infty(A)$. Now, $(\ker A)^\perp\cap C^\infty(A)$ is invariant under the action of polynomials of $A$, and therefore owing to (*) we deduce that $f^{[N]}-P_\mathcal{S}f^{[0]}\in(\ker A)^\perp$.

 Next, let us split
 \[
  P_\mathcal{S} f^{[N]}-P_\mathcal{S} f^{[0]}\;=\;(f^{[N]}-P_\mathcal{S}f^{[0]})-(f^{[N]}-P_\mathcal{S} f^{[N]})\,.
 \]
 Obviously, $P_\mathcal{S} f^{[N]}-P_\mathcal{S} f^{[0]}\in\ker A$. But in the right-hand side, as just shown, both $f^{[N]}-P_\mathcal{S}f^{[0]}\in(\ker A)^\perp$ and $f^{[N]}-P_\mathcal{S} f^{[N]}\in(\ker A)^\perp$. So $P_\mathcal{S} f^{[N]}-P_\mathcal{S} f^{[0]}\in(\ker A)^\perp$. The conclusion is necessarily $P_\mathcal{S} f^{[N]}-P_\mathcal{S} f^{[0]}=0$.
  
 This establishes \eqref{eq:PFNPF0}, by means of which formula (*) above takes also the form of \eqref{eq:iterativeNemi}.
\end{proof}

%

Let us now prove Lemmas \ref{lem:Abdd_Etheta} and \ref{lem:sigma-assumptions}.

\begin{proof}[Proof of Lemma \ref{lem:Abdd_Etheta}]
 Part (i) is evident from the fact that $\mathcal{D}(A^{\frac{\theta}{2}})=\cH$ for any $\theta\geqslant 0$, as $A$ is (everywhere-defined and) bounded and non-negative. 
 
 Part (ii) is therefore obvious if $\theta'\geqslant 0$. If, instead, $\theta\leqslant\theta'<0$, then $\ran (A^{-\frac{\theta}{2}})\subset \ran (A^{-\frac{\theta'}{2}})$, owing again to the boundedness and non-negativity of $A$, so part (ii) is actually valid in general.
 
 If $0\leqslant\theta\leqslant\theta'$, then
 \[
  u_{\theta'}(x)\;=\;A^{\theta'/2}(x-P_\mathcal{S}x)\;=\;A^{(\theta'-\theta)/2} A^{\theta/2}(x-P_\mathcal{S}x)\;=\;A^{(\theta'-\theta)/2}u_{\theta}(x)\,.
 \]
 If instead $\theta<0\leqslant\theta'$, then $u_{\theta'}(x)=A^{\theta'/2}(x-P_\mathcal{S}x)$ and $A^{-\theta/2}u_{\theta}(x)=x-P_\mathcal{S}x$, whence
 \[
  A^{(\theta'-\theta)/2}u_{\theta}(x)\;=\;A^{\theta'/2}(x-P_\mathcal{S}x)\;=\;u_{\theta'}(x)\,.
 \]
 Last, if $\theta\leqslant\theta'<0$, then $A^{-\xi/2}u_{\xi}(x)=x-P_\mathcal{S}x$ for both $\xi=\theta$ and $\xi=\theta'$, therefore from
 \[
  x-P_\mathcal{S}x\,=\,A^{-\theta/2}u_{\theta}(x)\,=\,A^{-\theta'/2}A^{(\theta'-\theta)/2}u_{\theta}(x)\quad\textrm{and}\quad A^{-\theta'/2}u_{\theta'}(x)\,=\,x-P_\mathcal{S}x
 \]
 one deduces that $u_{\theta'}(x)=A^{(\theta'-\theta)/2}u_{\theta}(x)$. In all possible cases the claimed identity is therefore proved. The inequality $\rho_{\theta'}(x)\:\leqslant\;\|A\|^{\theta'-\theta}\rho_{\theta}(x)$ then follows at once from \eqref{eq:defrhothetawithu}. This completes the proof of part (iii). 
\end{proof}

\begin{proof}[Proof of Lemma \ref{lem:sigma-assumptions}]
 Owing to \eqref{eq:iterativeNemi} and to the $A$-smoothness of $g$ and $f^{[0]}$, $f^{[N]}-P_\mathcal{S}f^{[N]}\in C^\infty(A)$, which by interpolation means in particular that $f^{[N]}-P_\mathcal{S}f^{[N]}\in \mathcal{D}(A^{\frac{\sigma}{2}})$ $\forall \sigma\geqslant 0$. This proves part (i) of the Lemma.
 
 Assume now that $f^{[0]}\in \mathscr{C}_{A,g}(\sigma)$ for some $\sigma<0$. In this case \eqref{eq:iterativeNemi} reads
 \[
  f^{[N]}-P_\mathcal{S}f^{[N]}\;=\;p_N(A)(f^{[0]}-P_\mathcal{S}f^{[0]})\;=\;p_N(A)\,A^{-\frac{\sigma}{2}}u_\sigma(f^{[0]})\,,
 \]
 thanks to the definition \eqref{eq:uthetadef} of $u_\sigma(f^{[0]})$. Therefore $f^{[N]}-P_\mathcal{S}f^{[N]}\in \ran(A^{-\frac{\sigma}{2}})$ and, again by \eqref{eq:uthetadef}, $ u_\sigma(f^{[N]})=p_N(A)\,u_\sigma(f^{[0]})$. This proves part (ii). 
\end{proof}

Next, let us establish an amount of important results that are measure-theoretic in nature. To this aim, with customary notation \cite{schmu_unbdd_sa}, let $E^A$ denote the projection-valued measure associated with the self-adjoint operator $A$, and let $\ud\langle x,E^{A}(\lambda)x\rangle$ denote the corresponding scalar measure associated with a vector $x\in\cH$. Such measures are supported on $\sigma(A)\subset[0,+\infty)$.

A special role is going to be played by the measure 
\begin{equation}\label{eq:defmu}
 \ud\mu_\sigma(\lambda)\;:=\;\ud\langle u_\sigma(f^{[0]}),E^{A}(\lambda)u_\sigma(f^{[0]})\rangle
\end{equation}
defined under the assumption that $f^{[0]}\in\mathscr{C}_{A,g}(\sigma)$ for a given $\sigma\in\mathbb{R}$.
Clearly, by definition, $\mu_\sigma$ is a \emph{finite} measure with
\begin{equation}\label{eq:mu-norm}
\int_{[0,+\infty)}\ud\mu_\sigma(\lambda)\;=\;\|u_\sigma(f^{[0]})\|^2\,.
\end{equation}

Two relevant properties of $\mu_\sigma$ are the following.

\begin{proposition}\label{prop:measuremu}
 For the given self-adjoint and non-negative operator $A$ on $\cH$, and for given $g\in C^\infty(A)$, $\sigma\in\mathbb{R}$, $f^{[0]}\in C^\infty(A)\cap\mathscr{C}_{A,g}(\sigma)$, consider the measure $\mu_\sigma$ defined by \eqref{eq:defmu}. Then:
 \begin{itemize}
  \item[(i)] one has
  \begin{equation}\label{eq:musigma}
   \ud\mu_\sigma(\lambda)\;=\;\lambda^\sigma\,\ud\langle f^{[0]}-P_\mathcal{S}f^{[0]},E^{A}(\lambda)(f^{[0]}-P_\mathcal{S}f^{[0]})\rangle\,;
  \end{equation}
  \item[(ii)] the spectral value $\lambda=0$ is not an atom for $\mu_\sigma$, i.e.,
  \begin{equation}\label{eq:noatomatzero}
   \mu_\sigma(\{0\})\;=\;0\,.
  \end{equation}
 \end{itemize} 
\end{proposition}

\begin{proof}
 The identity \eqref{eq:musigma} when $\sigma\geqslant 0$ follows immediately from the definition \eqref{eq:defmu} of $\ud\mu_\sigma$ and from the definition \eqref{eq:uthetadef} of $u_\sigma(f^{[0]})=A^{\frac{\sigma}{2}}(f^{[0]}-P_\mathcal{S}f^{[0]})$, owing to the property
 \[
  \ud\langle A^\alpha\psi,E^{A}(\lambda)A^\alpha\psi\rangle\;=\;\lambda^{2\alpha}\ud\langle\psi,E^{A}(\lambda)\psi\rangle\,,\qquad \alpha\geqslant 0\,,\qquad\psi\in\mathcal{D}(A^\alpha)\,.
 \]
 If instead $\sigma<0$, let us consider the auxiliary measures
 \[
  \ud\widetilde{\mu}_\sigma(\lambda)\;:=\;\lambda^{-\sigma}\ud\mu_\sigma(\lambda)\,,\qquad  \ud\widehat{\mu}_\sigma(\lambda)\;:=\;\ud\langle f^{[0]}-P_\mathcal{S}f^{[0]},E^{A}(\lambda)(f^{[0]}-P_\mathcal{S}f^{[0]})\rangle\,.
 \]
 On an arbitrary Borel subset $\Omega\subset [0,+\infty)$ one then has
 \[
  \begin{split}
   \widetilde{\mu}_\sigma(\Omega)\;&=\;\int_\Omega\lambda^{-\sigma}\ud\mu_\sigma(\lambda)\;=\;\|E^A(\Omega)A^{-\frac{\sigma}{2}}u_\sigma(f^{[0]})\|^2 \\
   &=\;\|E^A(\Omega)(f^{[0]}-P_\mathcal{S}f^{[0]})\|^2\;=\;\int_\Omega\ud\widehat{\mu}_\sigma(\lambda)\;=\;\widehat{\mu}_\sigma(\Omega)\,,
  \end{split}
 \]
 having used the definition \eqref{eq:uthetadef} in the form $A^{-\frac{\sigma}{2}}u_\sigma(f^{[0]})=f^{[0]}-P_\mathcal{S}f^{[0]}$. This shows that $\ud\widetilde{\mu}_\sigma(\lambda)=\ud\widehat{\mu}_\sigma(\lambda)$, whence again \eqref{eq:musigma}. Part (i) is proved.

 Concerning part (ii), let us recall from \eqref{eq:fNPSFN-kerperp} that $f^{[0]}-P_\mathcal{S}f^{[0]}\in(\ker A)^\perp$. Therefore, $\widehat{\mu}_\sigma(\{0\})=0$. Thus, \eqref{eq:musigma} implies that also $\mu_\sigma(\{0\})=0$.  
\end{proof}

In turn, Proposition \ref{prop:measuremu} allows us to discuss one further set of technical ingredients for the proof of Theorem \ref{thm:main}. They concern the polynomial $p_N$, in the expression \eqref{eq:iterativeNemi} of the $\xi$-iterates $f^{[N]}$, that corresponds to the actual minimisation \eqref{eq:thetaiterates}.

\begin{proposition}\label{prop:sN}
For the given self-adjoint and non-negative operator $A$ on $\cH$, and for given $g\in C^\infty(A)$, $\sigma\in\mathbb{R}$, $f^{[0]}\in C^\infty(A)\cap\mathscr{C}_{A,g}(\sigma)$, and $\xi\geqslant 0$ 
let
$f^{[N]}$ be the $N$-th $\xi$-iterate defined by \eqref{eq:thetaiterates} with initial guess $f^{[0]}$ and parameter $\theta=\xi$, and let
\begin{equation}\label{eq:defsN}
\begin{split}
 s_N\;:=&\:\;\arg\min_{\!\!\!\!\!\!\!\!\!\!p_N\in\mathbb{P}_N^{(1)}}\int_{[0,+\infty)}\lambda^\xi \,p_N^2(\lambda)\,\ud\langle f^{[0]}-P_\mathcal{S}f^{[0]},E^{A}(\lambda)(f^{[0]}-P_\mathcal{S}f^{[0]})\rangle 
\end{split} 
\end{equation}
for each $N\in\mathbb{N}$.
Then the following properties hold.
\begin{itemize}
 \item[(i)] One has
  \begin{equation}\label{eq:fNwithsigmaN}
  f^{[N]}-P_\mathcal{S}f^{[N]}\;=\;s_N(A)(f^{[0]}-P_\mathcal{S}f^{[0]}) \qquad \forall N\in\mathbb{N}\,.
 \end{equation}
 \item[(ii)] The family $(s_N)_{N\in\mathbb{N}}$ is a set of orthogonal polynomials on $[0,+\infty)$ with respect to the measure
 \begin{equation}\label{eq:defnumeasure}
  \begin{split}
   \ud\nu_{\xi}(\lambda)\;:=&\:\;\lambda^{\xi-\sigma+1}\,\ud\mu_\sigma(\lambda) \\
  =&\;\:\lambda^{\xi+1}\,\ud\langle f^{[0]}-P_\mathcal{S}f^{[0]},E^{A}(\lambda)(f^{[0]}-P_\mathcal{S}f^{[0]})\rangle
  \end{split}
  \end{equation}
 and satisfying
 \begin{equation}\label{eq:sNnormalisation}
  \deg s_N\;=\;N\,,\qquad s_N(0)\;=\;1\qquad\forall N\in\mathbb{N}
 \end{equation}
 (under the further tacit assumption that the $s_N$'s are all non-vanishing with respect to the measure $\mu_\sigma$).
 \item[(iii)] One has
 \begin{equation}\label{eq:rhosigmaint}
  \rho_\sigma(f^{[N]})\;=\;\int_{[0,+\infty)}s_N^2(\lambda)\,\ud\mu_\sigma(\lambda)\qquad \forall N\in\mathbb{N}\,.
 \end{equation}
\end{itemize}
\end{proposition}

\begin{proof}
 Denote temporarily by $\widetilde{s}_N\in\mathbb{P}_N^{(1)}$ the polynomial that qualifies the iterate $f^{[N]}$ in \eqref{eq:iterativeNemi} by means of the minimisation \eqref{eq:thetaiterates} with $\theta=\xi$. Then
 \[
 \begin{split}
  \min_{\!\!\!\!\!\!\! h\in\{f^{[0]}\}+\mathcal{K}_N(A,\mathfrak{R}_0)}\;&\|A^{\xi/2}(h-P_\mathcal{S}h)\|^2\;=\;\|A^{\xi/2}(f^{[N]}-P_\mathcal{S}f^{[N]})\|^2 \\
  &=\;\|A^{\xi/2}\widetilde{s}_N(A)(f^{[0]}-P_\mathcal{S}f^{[0]})\|^2 \\
  &=\;\int_{[0,+\infty)}\lambda^\xi \,\widetilde{s}_N^{\,2}(\lambda)\,\ud\langle f^{[0]}-P_\mathcal{S}f^{[0]},E^{A}(\lambda)(f^{[0]}-P_\mathcal{S}f^{[0]})\rangle\,.
 \end{split}
 \]
 Comparing the above identity with \eqref{eq:defsN} we see that $\widetilde{s}_N$ must be precisely the polynomial $s_N$. Therefore, \eqref{eq:iterativeNemi} takes the form \eqref{eq:fNwithsigmaN}. This proves part (i).

 By means of \eqref{eq:musigma} we may re-write \eqref{eq:defsN} as
 \[
  s_N\;=\;\arg\min_{\!\!\!\!\!\!\!\!\!\!p_N\in\mathbb{P}_N^{(1)}}\int_{[0,+\infty)}\lambda^{\xi-\sigma}\,p_N^2(\lambda)\ud\mu_\sigma(\lambda)\,.
 \]
 The latter minimising property of $s_N$ implies
 \[
 \begin{split}
  0\;&=\;\frac{\ud}{\ud\varepsilon}\Big|_{\varepsilon=0}\int_{[0,+\infty)}\lambda^{\xi-\sigma}\,(s_N(\lambda)+\varepsilon\lambda \,q_{N-1}(\lambda))^2\,\ud\mu_\sigma(\lambda) \\
  &=\;2\int_{[0,+\infty)}\lambda^{\xi-\sigma+1}\,s_N(\lambda)\,q_{N-1}(\lambda)\,\ud\mu_\sigma(\lambda)
 \end{split}
 \]
 for any $q_{N-1}\in \mathbb{P}_{N-1}$ (indeed, $s_N +\varepsilon\lambda \,q_{N-1}\in \mathbb{P}_N^{(1)}$). 
 Equivalently, owing to \eqref{eq:defnumeasure},
 \[
  \int_{[0,+\infty)}s_N(\lambda)\,q_{N-1}(\lambda)\,\ud\nu_{\xi}(\lambda)\;=\;0\qquad \forall q_{N-1}\in \mathbb{P}_{N-1}\,.
 \]
 Such a condition is valid for each $N\in\mathbb{N}$ and, as well known \cite{Szego-OrthPolyBook,Chihara-book-1978,Koornwinder_OrthPoly}, this amounts to saying that $(s_N)_{N\in\mathbb{N}}$ is a set of orthogonal polynomials on $[0,+\infty)$ with respect to the measure $\ud\nu_{\xi}$. Part (ii) is thus proved.

 If $\sigma\geqslant 0$, then \eqref{eq:PFNPF0}, \eqref{eq:rhoindicator}, \eqref{eq:musigma}, and \eqref{eq:fNwithsigmaN} yield
 \[
  \begin{split}
   \rho_\sigma(f^{[N]})\;&=\;\big\|A^{\frac{\sigma}{2}}(f^{[N]}-P_\mathcal{S}f^{[N]})\big\|^2\;=\;\big\|A^{\frac{\sigma}{2}}s_N(A)(f^{[0]}-P_\mathcal{S}f^{[0]})\big\|^2 \\
   &=\;\int_{[0,+\infty)}s_N^2(\lambda)\,\ud\mu_\sigma(\lambda)\,.
  \end{split}
 \]
 If instead $\sigma<0$, then owing to \eqref{eq:fNwithsigmaN} the identity \eqref{eq:usigma-uzero} reads
 \[
  u_\sigma(f^{[N]})\;=\;s_N(A)\,u_\sigma(f^{[0]})\,.
 \]
 The latter identity, together with \eqref{eq:rhoindicator} and \eqref{eq:defmu}, yield
 \[
   \rho_\sigma(f^{[N]})\;=\;\big\|u_\sigma(f^{[N]})\big\|^2\;=\;\big\|s_N(A)\,u_\sigma(f^{[0]})\big\|^2\;=\;\int_{[0,+\infty)}s_N^2(\lambda)\,\ud\mu_\sigma(\lambda)\,.
 \]
 In either case \eqref{eq:rhosigmaint} is established. This proves part (iii). 
\end{proof}

\begin{remark}
 The measure $\nu_{\xi}$ too is finite, with
 \begin{equation}
  \int_{[0,+\infty)}\ud\nu_{\xi}\;=\;\big\|A^{\frac{\xi+1}{2}}(f^{[0]}-P_\mathcal{S}f^{[0]})\big\|^2\,,
 \end{equation}
 as is evident from \eqref{eq:defnumeasure}. In fact, one could define $\nu_{\xi}$ for arbitrary $\xi\geqslant -1$: we keep the restriction to $\xi\geqslant 0$ because $\xi$ here is the parameter $\theta=\xi$ required in the definition \eqref{eq:thetaiterates} of the $\xi$-iterates, and as such must therefore be non-negative.  
\end{remark}

\begin{remark}\label{rem:implicit-xi-dependence}
 There is an implicit dependence on $\xi$ in each $s_N$, as is clear from \eqref{eq:defsN}, analogously to the fact that the iterates $f^{[N]}$'s depend on the choice of the parameter $\xi$. We simply omit such a dependence from the notation $s_N$.
\end{remark}

We thus see from Proposition \ref{prop:sN}(iii) that the control of the convergence of the $f^{[N]}$'s in the $\rho_\sigma$-sense is boiled down to monitoring a precise spectral integral, namely the right-hand side of \eqref{eq:rhosigmaint}. For an efficient estimate of the latter we shall make use of properties of the polynomials $s_N$ and of the measure $\nu_{\xi}$ that we are going to discuss in the remaining part of this Section.


 Here is our main result in this context.

\begin{proposition}\label{prop:properties_of_zeros}
 Consider the set $(s_N)_{N\in\mathbb{N}}$ of orthogonal polynomials on $[0,+\infty)$ with respect to the measure $\nu_{\xi}$, as defined in \eqref{eq:defsN} and \eqref{eq:defnumeasure} under the assumptions of Proposition \ref{prop:sN}.
 \begin{itemize}
  \item[(i)] For each $N\in\mathbb{N}$, either $s_N(\lambda)=0$ $\nu_{\xi}$-almost everywhere, or $s_N$ has exactly $N$ simple zeroes, all located in $(0,+\infty)$.
  \end{itemize}
  Assume now the $s_N$'s are all non-vanishing with respect to the $\nu_{\xi}$-measure, and denote by $\lambda_k^{(N)}$ the $k$-th zero of $s_N$, ordering the zeros as
  \begin{equation}\label{eq:EVordered}
   0<\lambda_1^{(N)}\,<\,\lambda_2^{(N)}\,<\,\cdots\,<\,\lambda_N^{(N)}\,.
  \end{equation}
  \begin{itemize}
  \item[(ii)] (Separation.) One has
  \begin{equation}
   \lambda_k^{(N+1)}\,<\,\lambda_k^{(N)}\,<\,\lambda_{k+1}^{(N+1)} \qquad \forall k\in\{1,2,\dots,N-1\}\,,
  \end{equation}
  that is, the zeroes of $s_N$ and $s_{N+1}$ mutually separate each other.
  \item[(iii)] (Monotonicity.) For each integer $k\geqslant 1$, 
  \begin{equation}
   \begin{split}
    (\lambda_k^{(N)})_{N=k}^\infty & \quad\textrm{is a decreasing sequence,} \\
        (\lambda_{N-k+1}^{(N)})_{N=k}^\infty & \quad\textrm{is an increasing sequence.}
   \end{split}
  \end{equation}
  In particular, the limits
    \begin{equation}\label{lambda1lambdainf}
   \lambda_1\;:=\;\lim_{N\to\infty}\lambda_1^{(N)}\,,\qquad \lambda_\infty\;:=\;\lim_{N\to\infty}\lambda_N^{(N)}
  \end{equation}
  exist in $[0,+\infty)\cup\{+\infty\}$.
  \item[(iv)] (Orthogonality.) One has
  \begin{equation}\label{eq:orthogonality}
   \int_{[0,\lambda_1^{(N)})}s^2_N(\lambda)\,\frac{\lambda_1^{(N)}}{\lambda_1^{(N)}-\lambda}\,\ud\nu_{\xi}(\lambda)\;=\;\int_{[\lambda_1^{(N)},+\infty)}s^2_N(\lambda)\,\frac{\lambda_1^{(N)}}{\lambda-\lambda_1^{(N)}}\,\ud\nu_{\xi}(\lambda)
  \end{equation}
  for any $N\in\mathbb{N}$.
 \end{itemize}
 Finally, assume in addition to the assumptions of Proposition \ref{prop:sN} also assumptions  \eqref{eq:maincond1}-\eqref{eq:maincond2} of Theorem \ref{thm:main}, or assumptions \eqref{eq:maincond3} of Corollary \ref{cor:main}. In other words, assume in addition that $f^{[0]},g\in\mathcal{D}^a(A)$, or also that $f^{[0]}-P_\mathcal{S}f^{[0]}\in\mathcal{D}^{qa}(A)$.
 \begin{itemize}
 \item[(v)] (Representation.) The measure $\nu_{\xi}$ is only supported on the so-called `true interval of orthogonality' $[\lambda_1,\lambda_\infty]$.
  Here and in the following, the symbol $[\lambda_1,\lambda_\infty]$ is understood as the closure of $(\lambda_1,\lambda_\infty)$.
 \end{itemize}
\end{proposition}

Observe that, for the first time, in Proposition \ref{prop:properties_of_zeros}(v) the assumption of $A$-analyticity of $g$ and $f^{[0]}$ kicks in, replacing the mere $A$-smoothness. This is the condition prescribed in the final Theorem \ref{thm:main}. So, prior to presenting the proof of Proposition \ref{prop:properties_of_zeros}, let us highlight in what form we shall exploit the extra condition of $A$-analyticity of $g$ and $f^{[0]}$.

\begin{lemma}\label{lem:fpPf-analitic}
 Let $A$ be a linear operator on a Hilbert space $\cH$.
 \begin{itemize}
  \item[(i)] Assume that $g\in\mathcal{D}^{qa}(A)\cap\mathrm{ran}A$. Then any $f\in\mathcal{D}(A)$ such that $Af=g$ satisfies $f\in \mathcal{D}^{qa}(A)$.
  \item[(ii)] Assume that $g\in\mathcal{D}^{a}(A)\cap\mathrm{ran}A$. Then any $f\in\mathcal{D}(A)$ such that $Af=g$ satisfies $f\in \mathcal{D}^{a}(A)$.
  \item[(iii)] Assume that $g\in\mathcal{D}^a(A)\cap\mathrm{ran}A$ and $f^{[0]}\in \mathcal{D}^a(A)$. Then $f^{[0]}-P_\mathcal{S}f^{[0]}\in \mathcal{D}^a(A)$.
 \end{itemize}
%
%
\end{lemma}

\begin{proof} 
 (i) As $Af=g$, then 
 \[
  \sum_{n=1}^\infty\|A^nf\|^{-\frac{1}{n}}=\|g\|^{-1}+\sum_{n=1}^\infty\|A^n g\|^{-\frac{1}{n+1}}\,.
 \]
 The latter series, by a standard ratio test (d'Alembert's criterion), is asymptotic to $\sum_{n=1}^\infty\|A^n g\|^{-\frac{1}{n}}$ and hence diverges because $g$ is quasi-analytic (see \eqref{def:quasianalytic-vectors} above). Then also $\sum_{n=1}^\infty\|A^nf\|^{-\frac{1}{n}}=+\infty$, whence the quasi-analyticity of $f$.

 (ii) As $Af=g$, then $A^{n-1} g=A^{n}f$ for any integer $n\geqslant 1$. By definition of $A$-analyticity of $g$ (see \eqref{def:analytic-vectors} above), there is $C_g>0$ such that
 \[
  \|A^n f\|\;=\;\|A^{n-1}g\|\;\leqslant\;C_g^{n-1}\,(n-1)!\;\leqslant\; D_f^n \,n!\,,\qquad n\geqslant 2\,,
 \]
 having set $D_f:=\max\{1,C_g\}$. The latter inequality is due to $D_f\geqslant C_g$, whence $D_f^{n-1}\geqslant C_g^{n-1}$, and to $D_f\geqslant 1$, whence $D_f^n\geqslant D_f^{n-1}$. As $\|Af\|=\|g\|$ (the $n=1$ case), then setting $C_f:=	\max\{\|g\|,D_f\}$ finally yields
 \[
  \|A^n f\|\;\leqslant\;C_f^n \,n!\,,\qquad n\geqslant 1\,,
 \]
 which in view of \eqref{def:analytic-vectors} expresses the $A$-analyticity of $f$.

 (iii) On account of part (ii), any solution $f$ belongs to $\mathcal{D}^a(A)$. Then in particular $P_\mathcal{S}f^{[0]}\in \mathcal{D}^a(A)$, and since $\mathcal{D}^a(A)$ is a linear subspace, the conclusion follows by linearity.
\end{proof}

\begin{lemma}\label{lem:Abeta-qa}
 Given $A=A^*\geqslant\mathbb{O}$,
 \begin{equation}\label{eq:qa-Abeta-invariant}
  A^{\beta}\mathcal{D}^{qa}(A)\;\subset\; \mathcal{D}^{qa}(A)\qquad \forall \beta\geqslant 0\,.
 \end{equation}
\end{lemma}

\begin{proof}
 We intend to apply this simple property (see, e.g., \cite[Lemma 7.17]{schmu_unbdd_sa}): 
\[\label{eq:propertystar}\tag{*}
\textrm{\begin{tabular}{p{10cm}}
if $S$ and $T$ are two densely defined operators with common domain $\mathcal{D}$ and such that $T\mathcal{D}\subset\mathcal{D}$, $S\mathcal{D}\subset\mathcal{D}$, and $TS=ST$ on $\mathcal{D}$, then $S\mathcal{D}^{qa}(T)\subset\mathcal{D}^{qa}(T)$.
\end{tabular}}
\]
 In the present case let us take
 \[
  S\;:=\;A^\beta\big|_{C^\infty(A)}\,,\qquad T\;:=\;A\big|_{C^\infty(A)}\,,\qquad \mathcal{D}\;:=\;C^\infty(A)\,.
 \]
 With this choice, obviously, $C^\infty(T)=C^\infty(A)$, whence also, owing to the definition \eqref{def:quasianalytic-vectors}, $\mathcal{D}^{qa}(T)=\mathcal{D}^{qa}(A)$. So, provided that all assumptions of \eqref{eq:propertystar} are matched, the conclusion $S\mathcal{D}^{qa}(T)\subset\mathcal{D}^{qa}(T)$ amounts precisely to \eqref{eq:qa-Abeta-invariant}. Concerning the assumptions of \eqref{eq:propertystar}, it is clear that both $T$ and $S$ are symmetric and densely defined, with common domain $\mathcal{D}$. The invariance properties $T\mathcal{D}\subset\mathcal{D}$ and $S\mathcal{D}\subset\mathcal{D}$ are tantamount as $A^\tau C^\infty(A)\subset C^\infty(A)$, respectively with $\tau=1$ and $\tau=\beta$, and in either case they follow from the fact that for every $h\in C^\infty(A)$ and any $k\in\mathbb{N}$, the vector $A^\tau h$ satisfies
 \[
  \begin{split}
   \|A^k A^\tau h\|^2\;&=\;\int_{[0,+\infty)}\!\lambda^{2(k+\tau)}\,\ud\mu_h^{(A)}(\lambda) \\
   &\leqslant\;\int_{[0,1)}\ud\mu_h^{(A)}(\lambda)+\int_{[1,+\infty)}\!\lambda^{2(k+\tau)}\,\ud\mu_h^{(A)}(\lambda) \\
   &\leqslant\;\|h\|^2+\|A^{k+\lceil \tau\rceil}h\|^2\;<\;+\infty\,,
  \end{split}
 \]
  where $\lceil \tau\rceil$ is the smallest integer greater than $\tau$. Last, the commutativity of $S$ and $T$ on $\mathcal{D}$ is obviously tantamount as $A A^\beta h=A^{1+\beta}h=A^\beta Ah$ for $h\in\mathcal{D}$. All assumptions of \eqref{eq:propertystar} are verified, and the Lemma is proved. 
\end{proof}

 On a related note, for completeness and later use, let us also recall this simple property.

 \begin{lemma}\label{lem:Aanalit-is-analit}
  For any operator $A$ on a Hilbert space $\cH$, $A\mathcal{D}^a(A)\subset \mathcal{D}^a(A)$
 \end{lemma}

 \begin{proof}
  Let $f\in \mathcal{D}^a(A)$ and $g:=Af$. Then
  \[
   \|A^n g\|\;=\;\|A^{n+1}f\|\;\leqslant\; C_f^{n+1}\,(n+1)!
  \]
  for some $C_f>0$ and for all $n\in\mathbb{N}_0$. Set $C_g:=2(\max\{C_f,1\})^2$ and take $n\in\mathbb{N}$. Then $C_f^{1+\frac{1}{n}}(1+n)^{\frac{1}{n}}\leqslant (\max\{C_f,1\})^2\cdot 2=C_g $, whence
  \[
   \|A^n g\|\;\leqslant\; C_f^{n+1}\,(n+1)!\;\leqslant\; C_g^n \,n!\qquad\forall n\in\mathbb{N}\,,
  \]
  which shows that $g\in \mathcal{D}^a(A)$.  
 \end{proof}

\begin{proof}[Proof of Proposition \ref{prop:properties_of_zeros}]
Part (i) is standard from the theory of orthogonal polynomials (see, e.g., \cite[Theorem 3.3.1]{Szego-OrthPolyBook} or \cite[Theorem 5.2]{Chihara-book-1978}), owing to the fact that the map
\[
 \mathbb{P}([0,+\infty))\,\ni p\;\longmapsto \int_{[0,+\infty)}\,p(\lambda)\,\ud\nu_{\xi}(\lambda)
\]
is a positive-definite functional on $\mathbb{P}([0,+\infty))$.

Part (ii) is another standard fact in the theory of orthogonal polynomials (see, e.g., \cite[Theorem 3.3.2]{Szego-OrthPolyBook} or \cite[Theorem I.5.3]{Chihara-book-1978}). Part (iii), in turn, is an immediate corollary of part (ii).

Part (iv) follows from the identity
\[
 \int_{[0,+\infty)}s_N(\lambda)\,q_{N-1}(\lambda) \, \mathrm{d}\nu_\xi(\lambda)\;=\;0\qquad \forall q_{N-1}\in\mathbb{P}_{N-1}
\]
(already considered in the proof of Proposition \ref{prop:sN}, as a consequence of the orthogonality of the $s_N$'s), when the explicit choice 
\[
 q_{N-1}(\lambda)\;:=\;\frac{\,\lambda_1^{(N)}\,s_N(\lambda)}{\lambda_1^{(N)}-\lambda}
\]
is made.

For Part (v) let us first recall \cite[Definition I.5.2]{Chihara-book-1978} that the true interval of orthogonality $[\lambda_1,\lambda_\infty]$ is the smallest closed interval containing all the zeroes $\lambda_k^{(N)}$, and moreover \cite[Theorem II.3.1]{Chihara-book-1978} there exists a measure $\eta$ on $[0,+\infty)$ supported only on $[\lambda_1,\lambda_\infty]$ such that the $s_N$'s remain orthogonal with respect to $\eta$ too and 
\[
 \mu_k\;:=\;\int_{[0,+\infty)}\lambda^k\,\ud\nu_{\xi}(\lambda)\;=\;\int_{[\lambda_1,\lambda_\infty]}	\lambda^k\,\ud\eta(\lambda)\,,\qquad \forall k\in\mathbb{N}_0\,.
\]
Such $\eta$-measure is actually a Stieltjes measure associated with a bounded, non-decreasing function $\psi$ obtained as point-wise limit of a sub-sequence of $(\psi_N)_{N\in\mathbb{N}}$, where
\[
 \psi_N(\lambda)\;:=\;
 \begin{cases}
  \;0\,, & \lambda<\lambda_1^{(N)}\,, \\
  A_1^{(N)}+\cdots+A_p^{(N)}\,, &\lambda\in[\lambda_p^{(N)},\lambda_{p+1}^{(N)})\quad\textrm{for}\quad p\in\{1,\dots, n-1\}\,, \\
  \;\mu_0\,, &\lambda\geqslant \lambda_N^{(N)}
 \end{cases}
\]
and $A_1^{(n)},\dots, A_N^{(n)}$ are positive numbers determined by the Gauss quadrature formula
\[
 \mu_k\;=\;\sum_{p=1}^N A_p^{(N)}(\lambda_p^{(N)})^k\,,\qquad \forall k\in\{0,1,\dots,2N-1\}\,.
\]

We want to show that $\nu_{\xi}=\eta$, i.e., that the Hamburger moment problem that guarantees that $(s_N)_{N\in\mathbb{N}}$ is an orthogonal system on $[0,+\infty)$ is \emph{uniquely} solved with the measure $\nu_{\xi}$. To this aim, let us re-write the \emph{even} moments of $\nu_{\xi}$ as
\[
 \begin{split}
   \mu_{2k}\;&=\;\int_{[0,+\infty)}\lambda^{2k}\lambda^{\xi+1}\,\ud\langle f^{[0]}-P_\mathcal{S}f^{[0]},E^{A}(\lambda)(f^{[0]}-P_\mathcal{S}f^{[0]})\rangle \\
   &=\;\big\|A^k A^{\frac{\xi+1}{2}}(f^{[0]}-P_\mathcal{S}f^{[0]})\big\|^2\;=\;\|A^k\phi\|^2\,,
 \end{split}
\]
 having set $\phi:=A^{\frac{\xi+1}{2}}(f^{[0]}-P_\mathcal{S}f^{[0]})$.
 The extra assumptions made for this part ensure that $f^{[0]}-P_\mathcal{S}f^{[0]}\in\mathcal{D}^a(A)$, on account of Lemma \ref{lem:fpPf-analitic}, or directly that $f^{[0]}-P_\mathcal{S}f^{[0]}\in\mathcal{D}^{qa}(A)$. As a consequence, owing to Lemma \ref{lem:Abeta-qa}, $\phi\in\mathcal{D}^{qa}(A)$. The quasi-analyticity of $\phi$ then implies (see \eqref{def:quasianalytic-vectors} above)
 \[
  \sum_{k=1}^\infty \mu_{2k}^{-\frac{1}{2k}}\;=\;\sum_{k=1}^\infty \|A^k\phi\|^{-\frac{1}{k}}\;=\;+\infty\,.
 \]
 Now, the divergence of the above series $\sum_{k=1}^\infty \mu_{2k}^{-\frac{1}{2k}}$ is a well-known sufficient condition (Carleman's criterion, see, e.g., \cite[Theorem I.10]{Shohat-Tamarkin_ProblemOfMoments1943}) for the uniqueness of the Hamburger moment problem's solution. 
 This shows that $\nu_{\xi}=\eta$, thus proving that $\nu_{\xi}$ is supported only on $[\lambda_1,\lambda_\infty]$.
\end{proof}

\begin{remark}\label{rem:implicit-xi-dependence-II}
 Analogously to what already observed in Remark \ref{rem:implicit-xi-dependence}, there is an implicit dependence on $\xi$ of all the zeroes $\lambda_k^{(N)}$. For a more compact notation, such a dependence is omitted.
\end{remark}


In view of Proposition \ref{prop:properties_of_zeros}(i), when the $s_N$'s are not identically zero we can explicitly represent
\begin{equation}\label{eq:representation_sNhatsN}
   s_N(\lambda)\;=\;\prod_{k=1}^N\bigg(1-\frac{\lambda}{\;\lambda_k^{(N)}}\bigg)\,.
\end{equation}

The integral \eqref{eq:orthogonality} is going to play a central role in the main proof, so the next technical result we need is the following efficient estimate of such a quantity.

\begin{lemma}\label{lem:xisigmaestimate}
 Consider the set $(s_N)_{N\in\mathbb{N}}$ of orthogonal polynomials on $[0,+\infty)$ with respect to the measure $\nu_{\xi}$, as defined in \eqref{eq:defsN} and \eqref{eq:defnumeasure} under the assumptions of Proposition \ref{prop:sN} and with the further restriction $\xi-\sigma+1\geqslant 0$. Assume that the $s_N$'s are non-zero polynomials with respect to the measure $\nu_{\xi}$. Then, for any $N\in\mathbb{N}$,
 \begin{equation}\label{eq:xisigmaestimate}
 \int_{(\lambda_1,\lambda_1^{(N)})}s^2_N(\lambda)\,\frac{\lambda_1^{(N)}}{\lambda_1^{(N)}-\lambda}\,\ud\nu_{\xi}(\lambda)\;\leqslant\;\mu_\sigma((\lambda_1,\lambda_1^{(N)}))\,\Big(\frac{\xi-\sigma+1}{\delta_N}\Big)^{\xi-\sigma+1}\,,
 \end{equation}
 where
 \begin{equation}\label{eq:defdeltaN}
  \delta_N\;:=\;\frac{1}{\;\lambda_1^{(N)}}+2\sum_{k=2}^N\frac{1}{\;\lambda_k^{(N)}}\,.
 \end{equation}
\end{lemma}

\begin{remark}
 Estimate \eqref{eq:xisigmaestimate} provides a $(\xi,\sigma)$-dependent bound on a quantity that is $\xi$-dependent only. This is only possible for a constrained range of $\sigma$, namely $\sigma\leqslant\xi+1$. 
\end{remark}

\begin{proof}[Proof of Lemma \ref{lem:xisigmaestimate}]
 For each $N\in\mathbb{N}$, the function 
  \[
  \begin{split}
  [0,\lambda_1^{(N)}]\;\ni\;\lambda\;\longmapsto\;a_N(\lambda)\;:=&\;\:\frac{\:\lambda_1^{(N)}\lambda^{\xi-\sigma+1}s_N^2(\lambda)\,}{\lambda_1^{(N)}-\lambda} \\
  =&\;\:\lambda^{\xi-\sigma+1}\bigg(1-\frac{\lambda}{\;\lambda_1^{(N)}}\bigg)\prod_{k=2}^N\bigg(1-\frac{\lambda}{\;\lambda_k^{(N)}}\bigg)^{\!2}
  \end{split}
 \]
 (where we used the representation \eqref{eq:representation_sNhatsN} for $s_N$) is non-negative, smooth, and such that $a_N(0)=a_N(\lambda_1^{(N)})=0$. Let $\lambda_N^*\in(0,\lambda_1^{(N)})$ be the point of maximum for $a_N$. Then $a_N'(\lambda_N^*)=0$, which after straightforward computations yields
 \[
  \xi-\sigma+1\;\geqslant\;\lambda_N^*\bigg(\frac{1}{\;\lambda_1^{(N)}}+2\sum_{k=2}^N\frac{1}{\;\lambda_k^{(N)}}\bigg)\;=\;\lambda_N^*\delta_N\,,
 \]
 whence also
 \[
  \lambda_N^*\;\leqslant\;\frac{\xi-\sigma+1}{\delta_N}\,.
 \]
 Moreover, $0\leqslant 1-\lambda/\lambda_k^{(N)}\leqslant 1$ for $\lambda\in[0,\lambda_1^{(N)}]$ and for all $k\in\{1,\dots,N\}$, as $\lambda_1^{(N)}$ is the smallest zero of $s_N$. Therefore,
 \[
  a_N(\lambda)\;\leqslant\;a_N(\lambda_N^*)\;\leqslant\;(\lambda_N^*)^{\xi-\sigma+1}\;\leqslant\;\Big(\frac{\xi-\sigma+1}{\delta_N}\Big)^{\xi-\sigma+1}\,,\qquad \lambda\in[0,\lambda_1^{(N)}]\,.
 \]
 We then conclude
 \[
 \begin{split}
  \int_{(\lambda_1,\lambda_1^{(N)})}s^2_N(\lambda)\,\frac{\lambda_1^{(N)}}{\lambda_1^{(N)}-\lambda}\,\ud\nu_{\xi}(\lambda)\;&=\;\int_{(\lambda_1,\lambda_1^{(N)})}a_N(\lambda)\,\ud\mu_{\sigma}(\lambda) \\
  &\leqslant\;\mu_\sigma((\lambda_1,\lambda_1^{(N)}))\,\Big(\frac{\xi-\sigma+1}{\delta_N}\Big)^{\xi-\sigma+1}\,,
 \end{split}
 \]
 which completes the proof. 
\end{proof}

\section{Proof of Theorem \ref{thm:main} and additional observations}\label{sec:mainproof}

Let us present in this Section the proof of our main statements, Theorem \ref{thm:main} and Corollary \ref{cor:main}, based on the intermediate results established in the previous Section.

Owing to Proposition \ref{prop:sN}, we have to control the behaviour for large $N$ of the quantity
 \[
 \rho_\sigma(f^{[N]})\;=\;\int_{[0,+\infty)}s_N^2(\lambda)\,\ud\mu_\sigma(\lambda)\,.
 \]

Obviously, in the following we assume that none of the polynomials $s_N$ vanish with respect to the measure $\nu_{\xi}$ previously introduced in \eqref{eq:defnumeasure}, for otherwise for some $N$ one would have $\rho_\sigma(f^{[N]})=0$ and therefore $f^{[N]}=P_\mathcal{S}f^{[0]}\in\mathcal{S}$ (see Remark \ref{rem:finite_steps}, or also \eqref{eq:fNwithsigmaN}), meaning that the conjugate gradient algorithm has come to convergence in a finite number of steps. The conclusion of Theorem \ref{thm:main} would then be trivially true.

Let us first observe, from the relation \eqref{eq:defnumeasure} between the measures $\mu_\sigma$ and $\nu_{\xi}$ and from the fact that the latter is supported on the true interval of orthogonality $[\lambda_1,\lambda_\infty]$ (Proposition \ref{prop:properties_of_zeros}(v)), that the measure $\mu_\sigma$ too is supported on such an interval. Thus, in practice,
 \begin{equation}\label{eq:inpractice}
 \rho_\sigma(f^{[N]})\;=\;\int_{[\lambda_1,\lambda_\infty]}s_N^2(\lambda)\,\ud\mu_\sigma(\lambda)\,.
 \end{equation}
(Let us recall that $[\lambda_1,\lambda_\infty]$ is a shorthand for the closure of $(\lambda_1,\lambda_\infty)$, even when $\lambda_\infty=+\infty$.)

It is convenient to split
\begin{equation}\label{eq:est0}
 \begin{split}
  \int_{[\lambda_1,\lambda_\infty]}&s_N^2(\lambda)\,\ud\mu_\sigma(\lambda)\;= \;\mu_\sigma(\{\lambda_1\}) s_N^2(\lambda_1)+\int_{(\lambda_1,\lambda_1^{(N)})}s_N^2(\lambda)\,\ud\mu_\sigma(\lambda)\\
  &\qquad\qquad\qquad+\int_{[\lambda_1^{(N)},+\infty)}s_N^2(\lambda)\,\ud\mu_\sigma(\lambda) \\
  &\leqslant\;\mu_\sigma(\{\lambda_1\}) s_N^2(\lambda_1)+\mu_\sigma((\lambda_1,\lambda_1^{(N)}))+\int_{[\lambda_1^{(N)},+\infty)}s_N^2(\lambda)\,\ud\mu_\sigma(\lambda)\,.
 \end{split}
\end{equation}
Here we used the bound $s_N^2(\lambda)\leqslant 1$, $\lambda\in[0,\lambda_1^{(N)})$, that is obvious from \eqref{eq:representation_sNhatsN}.

Next, let us show that
\begin{equation}\label{eq:est1}
 \int_{[\lambda_1^{(N)},+\infty)}s_N^2(\lambda)\,\ud\mu_\sigma(\lambda)\;\leqslant\;\frac{1}{\:(\lambda_1^{(N)})^{\xi-\sigma+1}}\int_{[0,\lambda_1^{(N)})}s^2_N(\lambda)\,\frac{\lambda_1^{(N)}}{\lambda_1^{(N)}-\lambda}\,\ud\nu_{\xi}(\lambda)\,.
\end{equation}

In fact, \eqref{eq:est1} is a consequence of the properties of $s_N$ discussed in Section \ref{sec:technical}. To see that, let us consider the inequality
\begin{equation}\label{eq:newest}
 \begin{split}
  1\;&\leqslant\;\Big(\frac{\lambda}{\:\lambda_1^{(N)}}\Big)^{\xi-\sigma}\;=\;\frac{1}{\:(\lambda_1^{(N)})^{\xi-\sigma+1}}\cdot\frac{\:\lambda_1^{(N)}}{\lambda}\cdot\lambda^{\xi-\sigma+1} \\
  &\leqslant\;\frac{1}{\:(\lambda_1^{(N)})^{\xi-\sigma+1}}\cdot\frac{\lambda_1^{(N)}}{\lambda-\lambda_1^{(N)}}\cdot\lambda^{\xi-\sigma+1}\qquad\qquad(\lambda\geqslant\lambda_1^{(N)})\,,
 \end{split}
\end{equation}
which is valid owing to the constraint $\xi-\sigma\geqslant 0$. Then,
%
%
%
%
%
\[
 \begin{split}
  \int_{[\lambda_1^{(N)},+\infty)}s_N^2(\lambda)\,\ud\mu_{\sigma}(\lambda)\;&\leqslant\;\frac{1}{\:(\lambda_1^{(N)})^{\xi-\sigma+1}}\int_{[\lambda_1^{(N)},+\infty)}s^2_N(\lambda)\,\frac{\lambda_1^{(N)}}{\lambda-\lambda_1^{(N)}}\,\ud\nu_{\xi}(\lambda) \\
  &=\;\frac{1}{\:(\lambda_1^{(N)})^{\xi-\sigma+1}}\int_{[0,\lambda_1^{(N)})}s^2_N(\lambda)\,\frac{\lambda_1^{(N)}}{\lambda_1^{(N)}-\lambda}\,\ud\nu_{\xi}(\lambda)\,,
 \end{split}
\]
having used \eqref{eq:defnumeasure} and \eqref{eq:newest} in the first step, and the orthogonality property \eqref{eq:orthogonality} in the second. Estimate \eqref{eq:est1} is thus proved.

In turn, from \eqref{eq:est1} one gets
\begin{equation}\label{eq:estnew1}
 \begin{split}
    \int_{[\lambda_1^{(N)},+\infty)}&s_N^2(\lambda)\,\ud\mu_\sigma(\lambda)\;\leqslant\;\frac{\,\lambda_1^{(N)}s^2_N(\lambda_1)\,}{\lambda_1^{(N)}-\lambda_1}\,\frac{\nu_{\xi}(\{\lambda_1\})}{\:(\lambda_1^{(N)})^{\xi-\sigma+1}} \\
    &\qquad\qquad +  \frac{1}{\:(\lambda_1^{(N)})^{\xi-\sigma+1}}\int_{(\lambda_1,\lambda_1^{(N)})}s^2_N(\lambda)\,\frac{\lambda_1^{(N)}}{\lambda_1^{(N)}-\lambda}\,\ud\nu_{\xi}(\lambda)\\
    &\leqslant \;\frac{\,\lambda_1^{(N)}s^2_N(\lambda_1)\,}{\lambda_1^{(N)}-\lambda_1}\,\mu_\sigma(\{\lambda_1\})+\Big(\frac{\xi-\sigma+1}{\lambda_1^{(N)}\delta_N}\Big)^{\xi-\sigma+1}\mu_\sigma((\lambda_1,\lambda_1^{(N)})) \\
    &\leqslant\;\frac{\,\lambda_1^{(N)}s^2_N(\lambda_1)\,}{\lambda_1^{(N)}-\lambda_1}\,\mu_\sigma(\{\lambda_1\})+(\xi-\sigma+1)^{\xi-\sigma+1}\mu_\sigma((\lambda_1,\lambda_1^{(N)}))\,,
 \end{split}
\end{equation}
where in the intermediate identity we used \eqref{eq:defnumeasure} to pass from $\nu_{\xi}$ to $\mu_\sigma$ and we applied Lemma \ref{lem:xisigmaestimate}, and in the final inequality we used the property $\lambda_1^{(N)}\delta_N\geqslant 1$ (following from \eqref{eq:defdeltaN}).

Thus, \eqref{eq:inpractice}, \eqref{eq:est0} and \eqref{eq:estnew1} yield
\begin{equation*}
\begin{split}
 \rho_\sigma(f^{[N]})\;&\leqslant\;\left(s^2_N(\lambda_1)+ \frac{\,\lambda_1^{(N)}s^2_N(\lambda_1)\,}{\lambda_1^{(N)}-\lambda_1}\right)\mu_\sigma(\{\lambda_1\}) \\
 &\qquad\qquad +\left(1+(\xi-\sigma+1)^{\xi-\sigma+1} \right)\mu_\sigma((\lambda_1,\lambda_1^{(N)}))\,,
\end{split}
\end{equation*}
whence also, using the factorisation \eqref{eq:representation_sNhatsN} for $s_N$,
\begin{equation}\label{eq:est3}
\begin{split}
 \rho_\sigma(f^{[N]})\;&\leqslant\;2\bigg(1-\frac{\lambda_1}{\lambda_1^{(N)}}\bigg)\prod_{k=2}^N\bigg(1-\frac{\lambda_1}{\;\lambda_k^{(N)}}\bigg)^2\mu_\sigma(\{\lambda_1\}) \\
 &\qquad\qquad +\left(1+(\xi-\sigma+1)^{\xi-\sigma+1} \right)\mu_\sigma((\lambda_1,\lambda_1^{(N)}))\,.
\end{split}
\end{equation}

In the right-hand side of \eqref{eq:est3} one has $\mu_\sigma((\lambda_1,\lambda_1^{(N)}))\xrightarrow[]{\;N\to\infty\;}0$. Moreover, depending on the value of $\lambda_1$, the quantity
\[
 \bigg(1-\frac{\lambda_1}{\lambda_1^{(N)}}\bigg)\prod_{k=2}^N\bigg(1-\frac{\lambda_1}{\;\lambda_k^{(N)}}\bigg)^2\mu_\sigma(\{\lambda_1\})
\]
either attains at every $N$ the value $\mu_\sigma(\{0\})$, if $\lambda_1=0$, and hence vanishes, owing to \eqref{eq:noatomatzero} from Proposition \ref{prop:measuremu}, or in general is bounded by
\[
 \frac{\lambda_1^{(N)}-\lambda_1}{\lambda_1^{(N)}}\|u_\sigma(f^{[0]})\|^2\,,
\]
owing to \eqref{eq:mu-norm} and to the ordering $0<\lambda_1^{(N)}<\lambda_2^{(N)}<\cdots<\lambda_N^{(N)}$ and $\lambda_1\leqslant\lambda_1^{(N)}$, and hence when $\lambda_1>0$ it vanishes in the limit $N\to\infty$.

In either case one concludes from \eqref{eq:est3} that $\rho_\sigma(f^{[N]})\xrightarrow[]{\;N\to\infty\;} 0$, thus completing the proof of Theorem \ref{thm:main}, and also of Corollary \ref{cor:main}, as our crucial Proposition \ref{prop:properties_of_zeros} is proved under the assumptions of either of them.

In the second part of this Section, we intend to highlight a number of important observations.

\begin{remark}
We see that the assumption that none of the polynomials $s_N$ vanish with respect to the measure $\nu_\xi$ in the proof of Theorem~\ref{thm:main} immediately excludes the possibility that $\lambda_1^{(N)} = \lambda_1$ for any $N \in \N$ by considerations in Proposition~\ref{prop:properties_of_zeros}(ii). Clearly $\lambda_1^{(N)} \neq 0$ for any $N \in \N$ too owing to Proposition~\ref{prop:properties_of_zeros}(i).
\end{remark}

\begin{remark}
 In retrospect, the assumption $\xi\geqslant\sigma$ was necessary to establish the bound \eqref{eq:est1} -- more precisely, the inequality \eqref{eq:newest}. In the step \eqref{eq:est3} (which is an application of Lemma \ref{lem:xisigmaestimate}), only the less restrictive assumption $\xi\geqslant\sigma-1$ was needed.
\end{remark}

\begin{remark}\label{rem:caseIcaseII}
Where exactly the true interval of orthogonality lies within $[0,+\infty)$ depends on the behaviour of the zeroes of the $s_N$'s. In particular, in terms of the quantity $\delta_N$ defined in \eqref{eq:defdeltaN} we distinguish two alternative scenarios:
\begin{itemize}
 \item[] \textsc{Case I:} $\delta_N\to\infty$ as $N\to\infty$;
 \item[] \textsc{Case II:} $\delta_N$ remains uniformly bounded, strictly above $0$, in $N$.
\end{itemize}

If the operator $A$ is bounded, then we are surely in Case I: indeed the orthogonal polynomials $s_N$ are defined on $\sigma(A)\subset[0,\|A\|]$, and their zeroes cannot exceed $\|A\|$: this forces $\delta_N$ to blow up with $N$. Moreover, $\lambda_\infty=\lim_{N\to\infty}\lambda_N^{(N)}<+\infty$.

If instead $A$ is unbounded, the $\lambda_k^{(N)}$'s fall in $[0,+\infty)$ and depending on their rate of possible accumulation at infinity $\delta_N$ may still diverge as $N\to\infty$ or stay bounded.

Clearly in Case II one has $\lambda_1>0$ and $\lambda_N=+\infty$, for otherwise the condition $\lambda_1=\lim_{N\to\infty}\lambda_1^{(N)}=0$ or $\lambda_\infty=\lim_{N\to\infty}\lambda_N^{(N)}<+\infty$ would necessarily imply $\delta_N\to +\infty$. Thus, in Case II the true interval of orthogonality is $[\lambda_1,+\infty)$ and it is separated from zero.
\end{remark}

\begin{remark} 
 Estimate \eqref{eq:est3}  in the proof and the reasoning thereafter show that the vanishing rate of $\rho_\sigma(f^{[N]})$ is actually controlled by the vanishing rate of the quantity $\mu_\sigma((\lambda_1,\lambda_1^{(N)}))$ if $\lambda_1=0$, or more generally of both quantities $(\lambda_1^{(N)}-\lambda_1)$ and $\mu_\sigma((\lambda_1,\lambda_1^{(N)}))$ if $\lambda_1>0$.  It is however unclear how to possibly quantify, in the above senses, the pace of $\lambda_1^{(N)}\to\lambda_1$. Let us recall (see Remark \ref{rem:Abdd} and \eqref{eq:Nem-conv-rate} in particular) that the Nemi\-rovskiy-Polyak analysis \cite{Nemirovskiy-Polyak-1985} for the bounded-$A$ case provides an explicit vanishing rate for $\rho_{\sigma'}(f^{[N]})$ for any $\sigma'\in(\sigma,\xi]$, based on a polynomial min-max argument that relies crucially on the \emph{finiteness} of the interval where the orthogonal polynomials $s_N$ are supported on (i.e., it relies on the boundedness of $\sigma(A)$). Therefore, there is certainly no room for applying the same argument to the present setting. In fact, we find it reasonable to expect that for generic (unbounded) $A$ the quantity $\rho_\sigma(f^{[N]})$ vanishes with arbitrarily slow pace depending on the choice of the initial guess $f^{[0]}$. A strong indication in this sense comes from the numerical tests discussed in Section \ref{sec:num_test}. 
\end{remark}

%
%

\begin{remark}\label{rem:unif-bdd}
 It is worth pointing out that removing from the hypotheses of Theorem \ref{thm:main} (respectively Corollary \ref{cor:main}) the $A$-analyticity of $g$ and $f^{[0]}$ (respectively, the quasi-analyticity assumptions \eqref{eq:maincond3} or \eqref{eq:maincond4}) and replacing it with just the minimal assumption of $A$-smoothness, one could have only come to the (still non-trivial, yet not-informative) conclusion that $\rho_\sigma(f^{[N]})\leqslant\kappa$ uniformly in $N$ for some $\kappa>0$. This is seen as follows. For sure, even if the moment problem for $\nu_\xi$ is indeterminate, the measure $\nu_\xi$ has some support within $[\lambda_1,\lambda_\infty]$ (see, e.g., \cite[Theorem II.3.2]{Chihara-book-1978}), and so does $\mu_\sigma$. However, in the lack of the information that $\mu_\sigma$ is \emph{only} supported in $[\lambda_1,\lambda_\infty]$, in the above proof one should additionally estimate, besides the vanishing quantity \eqref{eq:est0}, the extra term
 \[
   \int_{[0,\lambda_1)}s_N^2(\lambda)\,\ud\mu_\sigma(\lambda)\,.
 \]
 On account of the inequalities $\lambda_1\leqslant\lambda_1^{(N)}$ (Proposition \ref{prop:properties_of_zeros}(iii)) and $s_N(\lambda)\leqslant 1$ $\forall\lambda\in[0,\lambda_1)$ (representation \eqref{eq:representation_sNhatsN}), the above integral is controlled by $\int_{[0,\lambda_1)}\ud\mu_\sigma(\lambda)$, and is therefore bounded uniformly in $N$. 
\end{remark}

\begin{remark}[\textbf{Comparison with the proof of \cite{Nemirovskiy-Polyak-1985} valid for bounded $A$}]\label{rem:comparison1}~

\noindent Our proof generalises the Nemi\-rovskiy-Polyak analysis \cite{Nemirovskiy-Polyak-1985} with a crucial technical novelty that is necessary when $A$ is unbounded, and in fact it also yields a subtle improvement of the old argument for the bounded case.
More precisely, in \cite{Nemirovskiy-Polyak-1985} one does \emph{not} make use of the very useful property that $\mu_\sigma$ is only supported on $[\lambda_1,\lambda_\infty]$,  which is the outcome of the somewhat laborious path that led to Proposition \ref{prop:properties_of_zeros}(v) here.
The sole measure-theoretic information used in \cite{Nemirovskiy-Polyak-1985} is that $\lambda=0$ is not an atom for $\mu_\sigma$. Then in \cite{Nemirovskiy-Polyak-1985}, instead of naturally splitting the integration as in \eqref{eq:est0} above, one separates the small and the large spectral values at a threshold $\gamma_N=\min\{\lambda_1^{(N)},\delta_N^{-1/2}\}$. Clearly $\gamma_N\to 0$, because $\delta_N\to +\infty$ since  $A$ is bounded (see Remark \ref{rem:caseIcaseII} above), and through a somewhat lengthy analysis of the integration for $\lambda<\gamma_N$ and $\lambda\geqslant\gamma_N$ one reduces both integrations to one over $[0,\gamma_N)$. Then one finally pulls out the upper bound $\mu_\sigma([0,\gamma_N))$, which vanishes as $N\to\infty$ precisely because $\mu_\sigma$ is atom-less at $\lambda=0$.
In the unbounded case such a scheme cannot work: $\delta_N$ does not necessarily diverge and only the information that $\mu_\sigma$ is supported at the right of, and possibly at, $\lambda_1$ makes the final estimate meaningful. Furthermore, in retrospect, by splitting the integration as in \eqref{eq:est0} and not in the old manner of \cite{Nemirovskiy-Polyak-1985}, our proof shortens the overall argument and applies both to the bounded and to the unbounded case, with no need to introduce the $\gamma_N$ cut-off.
\end{remark}

\begin{remark}[\textbf{Continuation: comparison with subsequent surveys of \cite{Nemirovskiy-Polyak-1985}}]\label{rem:comparison2} The analysis of conjugate gradients in the bounded case is nicely revisited by Hanke in the monograph \cite{Hanke-ConjGrad-1995}, both by presenting a version of the same Nemi\-rovskiy-Polyak $\gamma_N$-argument \cite{Nemirovskiy-Polyak-1985}, and by relying on a dominated convergence argument for a choice of a sequence of polynomials that vanishes point-wise over $(0,1]$ or on a Banach-Steinhaus uniform boundedness argument. For obvious reasons, none of such schemes are exportable to the present unbounded setting by merely updating the assumption of $g$ and $f^{[0]}$ so as to be $A$-smooth: in order to deal with a measure supported on $[\lambda_1,\lambda_\infty]$, an interval that is possibly infinite and separated from zero, the additional measure-theoretic analysis of Proposition \ref{prop:properties_of_zeros} is needed. (One should not be misled when in \cite{Hanke-ConjGrad-1995} certain spectral integrals appear to run over the whole positive half-line: it is clear from the discussion therein that the zeroes of the considered orthogonal polynomials only fall within a \emph{bounded} interval.) Of course, as commented already, in the present generalised scheme one pays the price that any quantitative bound on the rate of convergence is lost.
\end{remark}

We conclude this Section with one further important fact we had alluded to in Remark \ref{rem:generic-regime}.

\begin{proposition}\label{prop:generic-regime}
 Let $\xi\geqslant 0$, and with respect to the Hilbert space $\cH=L^2(\mathbb{R},\ud x)$ let $A$ be the self-adjoint multiplication by $x^2$, and let $g:=x^2 f$ with
 \begin{equation}
  f(x)\;:=\;\frac{\sqrt{2}\,\mathbf{1}_{\mathbb{R}^+}(x)}{\,x^{\frac{3}{2}+\xi} (2\pi)^{\frac{1}{4}}}\,e^{-\frac{1}{4}(\log x^2)^2}\,.
 \end{equation}
 Set further $f^{[0]}\equiv 0$. Then:
  \begin{itemize}
   \item[(i)] $A$ is non-negative and $g\in C^\infty(A)\cap\ran(A)$;
   \item[(ii)] Neither $f$ nor $g$ are quasi-analytic for $A$;
   \item[(iii)] $P_\mathcal{S}f^{[0]}=f$;
   \item[(iv)] the measure $\nu_\xi$ defined in \eqref{eq:defnumeasure} is a log-normal distribution, i.e.,
   \[
    \ud\nu_{\xi}(\lambda)\;=\;\frac{\mathbf{1}_{\mathbb{R}^+}(\lambda)}{\,\lambda \sqrt{2\pi}\,}\,e^{-\frac{1}{2}(\log \lambda)^2}\,\ud\lambda\,,\qquad \lambda\geqslant 0\,.
   \]
   \item[(v)] The Hamburger moment problem for $\nu_\xi$ is indeterminate.
  \end{itemize}
\end{proposition}

Proposition \ref{prop:generic-regime} shows that in general, when $g$ and $f^{[0]}$ are only assumed to be $A$-smooth, and the vector $f^{[0]}-P_\mathcal{S}f^{[0]}$ is not necessarily quasi-analytic for $A$, the measure $\nu_\xi$ may fail to be supported \emph{entirely} in the true interval of orthogonality $[\lambda_1,\lambda_\infty]$. (Moreover, let us recall -- see, e.g., \cite[Exercise II.5.7]{Chihara-book-1978} -- that in the lack of unique solution to the moment problem discussed in the proof of Proposition \ref{prop:properties_of_zeros}, at least one representative measure has a part of its support \emph{outside} $[\lambda_1,\lambda_\infty]$.) As a consequence, the quantity $\rho_\sigma(f^{[N]})$, while staying uniformly bounded (Remark \ref{rem:unif-bdd}) is not guaranteed to vanish as $N\to\infty$.

\begin{proof}[Proof of Proposition \ref{prop:generic-regime}]
 The facts that $A\geqslant\mathbb{O}$ and $g\in\ran(A)$ (provided that $f\in\mathcal{D}(A)$) are obvious. Let us prove that $f\in C^\infty(A)$ (whence $f\in\mathcal{D}(A)$ and $g\in C^\infty(A)$). For $n\in\mathbb{N}_0$, and with the change of variable $y:=\log x^2$, we compute 
 \[
  \begin{split}
   \|A^n f\|_{L^2}^2\;&=\;\|x^{2n}f\|_{L^2}\;=\;\frac{2}{\sqrt{2\pi}}\int_0^{+\infty} x^{4n-3-2\xi}\,\,e^{-\frac{1}{2}(\log x^2)^2}\,\ud x \\
   &=\;\frac{1}{\sqrt{2\pi}}\int_\mathbb{R} e^{(2n-1-\xi)y}\,e^{-\frac{1}{2} y^2}\,\ud y\;=\;e^{\frac{1}{2}(2n-1-\xi)^2}\;<\;+\infty\,.
  \end{split}
 \]
 Thus, $f\in C^\infty(A)$ and part (i) is proved.

 From the latter computation we also find
 \[
  \sum_{n\in\mathbb{N}}\|A^n f\|_{L^2}^{-\frac{1}{n}}\;=\;\sum_{n\in\mathbb{N}}e^{-\frac{1}{4n}(2n-1-\xi)^2}\;<\;+\infty\,,
 \]
 whence $f\notin\mathcal{D}^{qa}(A)$ (on account of definition \eqref{def:quasianalytic-vectors}), and also $g\notin\mathcal{D}^{qa}(A)$. Thus, (ii) is proved.
 
 Part (iii) follows from $Af=g$ and from the injectivity of $A$.
 
 Concerning part (iv), let us observe first of all that the spectral measure of $A$ is only supported on $\sigma(A)=[0,+\infty)$, and moreover (see, e.g., \cite[Example 5.3]{schmu_unbdd_sa}), the spectral projections $E^{(A)}(\Omega)$, for given Borel subset $\Omega\subset[0,+\infty)$, are nothing but the multiplication operators by the characteristic functions $\mathbf{1}_{a^{-1}(\Omega)}$, where $x\mapsto a(x):=x^2$. Thus,
 \[
  \begin{split}
  \langle f,E^{A}(\Omega)f\rangle\;&=\;\int_{a^{-1}(\Omega)}|f(x)|^2\,\ud x\;=\;\int_{\{\lambda\in\mathbb{R}\,|\,\lambda^2\in\Omega\}}|f(x)|^2\,\ud x \\
  &=\;\int_{\{\sqrt{\lambda}\,|\,\lambda\in\Omega\}}|f(x)|^2\,\ud x+\int_{\{-\sqrt{\lambda}\,|\,\lambda\in\Omega\}}|f(x)|^2\,\ud x\,.
  \end{split}
 \]
 Therefore, with $\Omega=[0,\lambda]$ and $E^{(A)}(\lambda)\equiv E^{(A)}([0,\lambda])$ for $\lambda\geqslant 0$,
 \[
  \begin{split}
   \frac{\ud \langle f,E^{A}(\lambda)f\rangle}{\ud \lambda}\;&=\;\frac{\ud}{\ud\lambda}\int_{-\sqrt{\lambda}}^{\sqrt{\lambda}}|f(\lambda)|^2\ud \lambda\;=\;\frac{1}{\,2\sqrt{\lambda}\,}\big|f(\sqrt{\lambda})\big|^2+\frac{1}{\,2\sqrt{\lambda}\,}\big|f(-\sqrt{\lambda})\big|^2 \\
   &=\;\frac{1}{\,2\sqrt{\lambda}\,}\big|f(\sqrt{\lambda})\big|^2\;=\;\frac{\mathbf{1}_{\mathbb{R}^+}(\lambda)}{\,\lambda^{2+\xi} \sqrt{2\pi}\,}\,e^{-\frac{1}{2}(\log \lambda)^2}\,.
  \end{split}
 \]
  From this, from part (iii), and from \eqref{eq:defnumeasure},
 \[
  \ud\nu_{\xi}(\lambda)\;=\;\lambda^{\xi+1}\,\ud\langle f,E^{A}(\lambda)f\rangle\;=\;\frac{\mathbf{1}_{\mathbb{R}^+}(\lambda)}{\,\lambda \sqrt{2\pi}\,}\,e^{-\frac{1}{2}(\log \lambda)^2}\,\ud\lambda\,.
 \]
  Part (iv) is proved.

  Last, concerning part (v), it is well known (see, e.g., \cite[Exercise 8.7]{Sullivan-UncertQuant}) that the space of polynomials on $[0,+\infty)$ is \emph{not dense} in $L^2([0,+\infty),\ud\nu_\xi)$ when the measure $\nu_\xi$, as in the present case, is a log-normal distribution. As an immediate consequence the solution for the Hamburger moment problem for $\nu_\xi$ is not unique: for any $\nu_\xi$-integrable function $\varphi$ that is $\nu_\xi$-orthogonal to the subspace of polynomials, the distinct measures $\nu_\xi$ and $(1+\varphi)\nu_\xi$ have obviously the same moments (see, e.g., \cite[Corollary 2.3.3]{Akhiezer-ClassicalMomentProb} for the same conclusion in abstract terms).
\end{proof}

\section{Numerical tests}\label{sec:num_test}


In this Section we discuss a selection of numerical tests that we run with the three-fold purpose of confirming the main features of our convergence result, corroborating our intuition on certain relevant differences with respect to the bounded case, and exploring the behaviour of the unbounded conjugate gradient algorithm beyond the regime covered by our main theorem.

We choose $\cH=L^2(\mathbb{R})$ and 
\begin{equation}
 \begin{split}
  \textrm{test-1a:} & \qquad A=-\frac{\ud^2}{\ud x^2}+\mathbbm{1}\,,\qquad\mathcal{D}(A)=H^2(\mathbb{R})\,,\qquad f(x)=e^{-x^2}\,, \\
  \textrm{test-1b:} & \qquad A=-\frac{\ud^2}{\ud x^2}\,,\qquad\quad\;\;\,\mathcal{D}(A)=H^2(\mathbb{R})\,,\qquad f(x)=e^{-x^2}\,, \\
  \textrm{test-2a:} & \qquad A=-\frac{\ud^2}{\ud x^2}+\mathbbm{1}\,,\qquad\mathcal{D}(A)=H^2(\mathbb{R})\,,\qquad f(x)=(1+x^2)^{-1}\,, \\
  \textrm{test-2b:} & \qquad A=-\frac{\ud^2}{\ud x^2}\,,\qquad\quad\;\;\,\mathcal{D}(A)=H^2(\mathbb{R})\,,\qquad f(x)=(1+x^2)^{-1}\,,
 \end{split}
\end{equation}
where $H^2$ denotes the usual Sobolev space of second order. In either case $A$ is an unbounded, injective, non-negative, self-adjoint operator on $\cH$; but only in tests 1a and 2a does $A^{-1}$ exist as an everywhere defined bounded operator.

We then consider the inverse linear problem $Af=g$ with the datum $g\in\mathrm{ran}A$ given by the above explicit choice of the solution $f$, and we construct conjugate gradient approximate solutions $f^{[N]}$ to $f$, namely $\xi$-iterates with $\xi=1$, with initial guess $f^{[0]}=0$ (the zero function on $\mathbb{R}$). Thus, each $f^{[N]}$ is searched for over the Krylov subspace $\mathcal{K}_N(A,g)=\mathrm{span}\{g,Ag,\dots,A^{N-1}g\}$. $f^{[0]}$ is trivially smooth, and so are $f$ and $g$, therefore the $1$-iterates are all well-defined. Owing to the injectivity of $A$ in all considered cases, necessarily $P_\mathcal{S}f^{[0]}=f$. The algorithm is well defined in all tests, as $f$ (and hence $g$) is a smooth function and is square-integrable together with all its derivatives (i.e., $g\in C^\infty(A)$).

Of course in practice we replace the minimisation \eqref{eq:thetaiterates} with the standard, equivalent algebraic construction for the $f^{[N]}$'s \cite{Saad-2003_IterativeMethods,Liesen-Strakos-2003}, so as to implement it as a routine in a symbolic computation software.

Iteratively we evaluate
\begin{equation}
 \begin{split}
  \rho_0(f^{[N]})\;&=\;\big\|f^{[N]}-f\big\|^2 \\
  \rho_1(f^{[N]})\;&=\;\big\langle f^{[N]}-f,A\,(f^{[N]}-f)\big\rangle \\
  \rho_2(f^{[N]})\;&=\;\big\|Af^{[N]}-g\big\|^2\,,
 \end{split}
\end{equation}
(see \eqref{eq:xi123} above) and we monitor the behaviour of such three quantities as $N$ increases.

The choice of the data $g$ in our tests is made so as tests 1 fall within the scope of our Theorem \ref{thm:main}, whereas tests 2 go \emph{beyond} it. Indeed:

\begin{lemma}\label{lem:anal-non-quasian} With respect to the Hilbert space $L^2(\mathbb{R},\ud x)$ and the self-adjoint operators $A=-\frac{\ud^2}{\ud x^2}+1$ or $A=-\frac{\ud^2}{\ud x^2}$ introduced above,
\begin{itemize}
 \item[(i)] the function $f=e^{-x^2}$, and hence $g=Af$ is analytic;
 \item[(ii)] the function $f=(1+x^2)^{-1}$, and hence $g=Af$ is not quasi-analytic. 
\end{itemize}
\end{lemma}

\begin{proof}
 (i) It suffices to show that $f$ is $A$-analytic, the same conclusion for $g$ then follows from Lemma \ref{lem:Aanalit-is-analit}. By a standard criterion, $A$-analyticity of $f$ is tantamount as $f\in\mathcal{D}(e^{z T})$ for all $z\in\mathbb{C}$ with $|\mathfrak{Re}\,z|$ small enough (see, e.g., \cite[Corollary 7.9]{schmu_unbdd_sa}). In terms of the Hilbert space isomorphism $L^2(\mathbb{R},\ud x)\xrightarrow{\cong}L^2(\mathbb{R},\ud p)$, $h\mapsto\widehat{h}$ induced by the Fourier transform, when $f=e^{-x^2}$ one has $\widehat{f}=\frac{1}{\sqrt{2}}e^{-\frac{1}{4}p^2}$ and the considered $A$'s are unitarily equivalent to the multiplication by $p^2+1$, or by $p^2$. Therefore,
 \[
  \widehat{e^{z A}f}\;=\;\frac{1}{\sqrt{2}}\,e^{z(p^2+1)}\,e^{-\frac{1}{4}p^2}\,\in\, L^2(\mathbb{R},\ud p)\qquad \textrm{for } |\mathfrak{Re}\,z|<\frac{1}{4}\,,
 \]
 and the same conclusion holds with $e^{z\,p^2}$ in place of $e^{z(p^2+1)}$, so both $A$'s are covered. Thus, $f=e^{-x^2}$ is indeed $A$-analytic. Incidentally the same criterion also shows that $f=(1+x^2)^{-1}$ is not analytic, because $\widehat{f}=\sqrt{\frac{\pi}{2}}\,e^{-|p|}$, and for no non-zero values of $\mathfrak{Re} z$ is the function 
 \[
  \widehat{e^{z A}f}\;=\;\sqrt{\frac{\pi}{2}}\,e^{z(p^2+1)}\,e^{-|p|}
 \]
 square-integrable on $\mathbb{R}$.
 
 (ii) It suffices to show that $f$ is not quasi-analytic for $A$, the same conclusion for $g$ then follows from Lemma \ref{lem:fpPf-analitic}(i). Besides, it suffices to show the lack of quasi-analyticity of $f$ with respect to $A=-\frac{\ud^2}{\ud x^2}$: then, since for $n\in\mathbb{N}$
 \[
  \begin{split}
   \|\textstyle{(-\frac{\ud^2}{\ud x^2}+\mathbbm{1})^n}f\|_{L^2}^2\;&=\;\langle f,\textstyle{(-\frac{\ud^2}{\ud x^2}+\mathbbm{1})^{2n}}f\rangle_{L^2}\;\geqslant\;\langle f,\textstyle{(-\frac{\ud^2}{\ud x^2})^{2n}}f\rangle_{L^2} \\
   &=\;\|\textstyle{(-\frac{\ud^2}{\ud x^2})^n}f\|_{L^2}^2\,,
  \end{split}
 \]
 one concludes
 \[
  \sum_{n\in\mathbb{N}}\|
  {\textstyle{(-\frac{\ud^2}{\ud x^2}+\mathbbm{1})^n}}
  f\|_{L^2}^{-\frac{1}{n}}\;\leqslant\;\sum_{n\in\mathbb{N}}\|\textstyle{(-\frac{\ud^2}{\ud x^2})^n}f\|_{L^2}^{-\frac{1}{n}}\;<\;+\infty
 \]
 (last inequality following from \ref{def:quasianalytic-vectors} and the fact that $f$ is not quasi-analytic for $-\frac{\ud^2}{\ud x^2}$), that is, the lack of quasi-analyticity of $f$ also for $-\frac{\ud^2}{\ud x^2}+\mathbbm{1}$. So now let $A=-\frac{\ud^2}{\ud x^2}$, $f=(1+x^2)^{-1}$, and using $\widehat{f}=\sqrt{\frac{\pi}{2}}\,e^{-|p|}$ we compute
 \[
  \begin{split}
  \|A^n f\|^2_{L^2}\;&=\;\int_\mathbb{R} \Big(\sqrt{\frac{\pi}{2}}\,(p^2)^ne^{-|p|}\Big)^2\,\ud p   \;=\;\pi\int_0^{+\infty}p^{4n}\,e^{-2 p}\,\ud p \\
  &=\;\pi\,\frac{\Gamma(1+4n)}{2^{1+4n}}\,.
  \end{split}
 \]
 Using the known asymptotics \cite[Eq.~(6.1.37)]{Abramowitz-Stegun-1964} of the gamma function
 \[
   \Gamma(t)\stackrel{t\to +\infty}{=}\sqrt{2\pi}\,e^{-t}\,t^{t-\frac{1}{2}}\,(1+O(t^{-1}))\,,
 \]
 we obtain, asymptotically as $n\to\infty$,
 \[
  \|A^n f\|^2_{L^2}\;=\;\frac{\pi\sqrt{2\pi}}{e}\,e^{-4n}\,(1+4n)^{\frac{1}{2}+4n}\,2^{-(1+4n)}\,(1+O(n^{-1}))\,,
 \]
 whence
 \[
  \|A^n f\|^{-\frac{1}{n}}_{L^2}\;=\;4 e^2 n^{-2}\,(1+O(n^{-1}))\,.
 \]
 This shows that the series $\sum_{n\in\mathbb{N}} \|A^n f\|^{-\frac{1}{n}}_{L^2}$ is asymptotics to $\sum_{n\in\mathbb{N}} n^{-2}$ and therefore converges. $f$ is not quasi-analytic for $A$.
\end{proof}

On account of Lemma \ref{lem:anal-non-quasian}, tests 1 are covered by our theoretical analysis: since obviously $f^{[0]}\in\mathscr{C}_{A,g}(\sigma)$ $\forall\sigma\geqslant 0$, then Theorem \ref{thm:main} ensures that $\rho_\sigma(f^{[N]})\to 0$ for any $\sigma\in[0,1]$. In particular, both the error ($\rho_0$) and the energy norm ($\rho_1$) are predicted to vanish as $N\to\infty$.

In the bounded case also the residual ($\rho_2$) would automatically vanish (Remark \ref{rem:Abdd}), but in tests-1 this indicator is not controlled by Theorem \ref{thm:main} and it is worth monitoring it.

A fourth meaningful quantity to monitor is $N^2\rho_1(f^{[N]})$. Recall indeed that \emph{if $A$ was bounded} the energy norm would be predicted to vanish not slower than a rate of order $N^{-2}$ (as given by \eqref{eq:Nem-conv-rate} with $\xi=1$, $\sigma=0$, $\sigma'=1$). Thus, detecting now the possible failure of $N^2\rho_1(f^{[N]})$ to stay bounded uniformly in $N$ is an immediate signature of the fact that one cannot apply to the unbounded-$A$ scenario the `classical' quantitative convergence rate predicted by Nemirovskiy and Polyak for the bounded-$A$ scenario \cite{Nemirovskiy-Polyak-1985}, which in fact was also proved to be optimal in that case \cite{Nemirovskiy-Polyak-1985-II}.

The results of tests 1a and 1b are shown respectively in Figure \ref{fig:1} and \ref{fig:2}.

\begin{figure}
\begin{minipage}{\textwidth}
\includegraphics[width=0.45\textwidth]{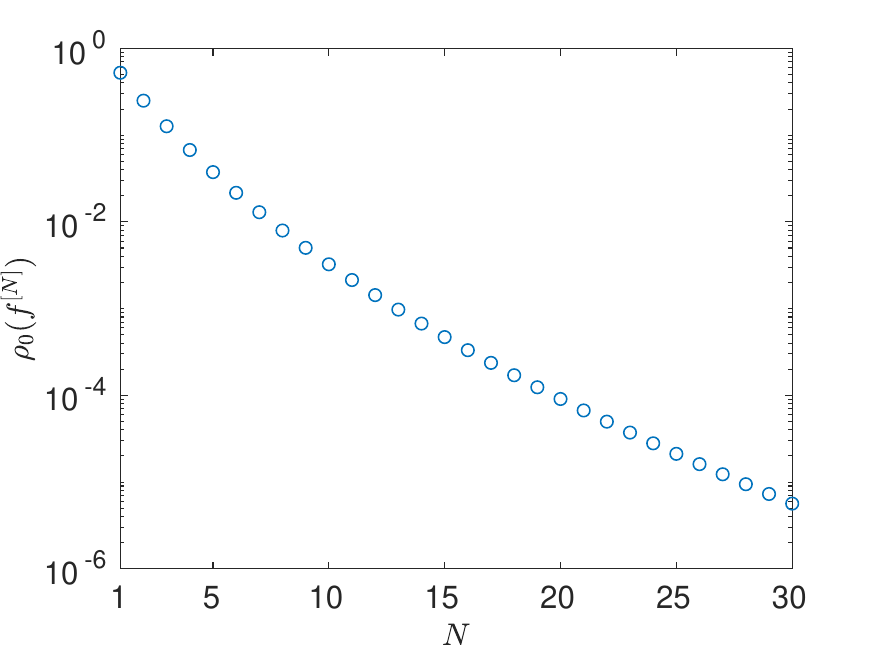}
\includegraphics[width=0.45\textwidth]{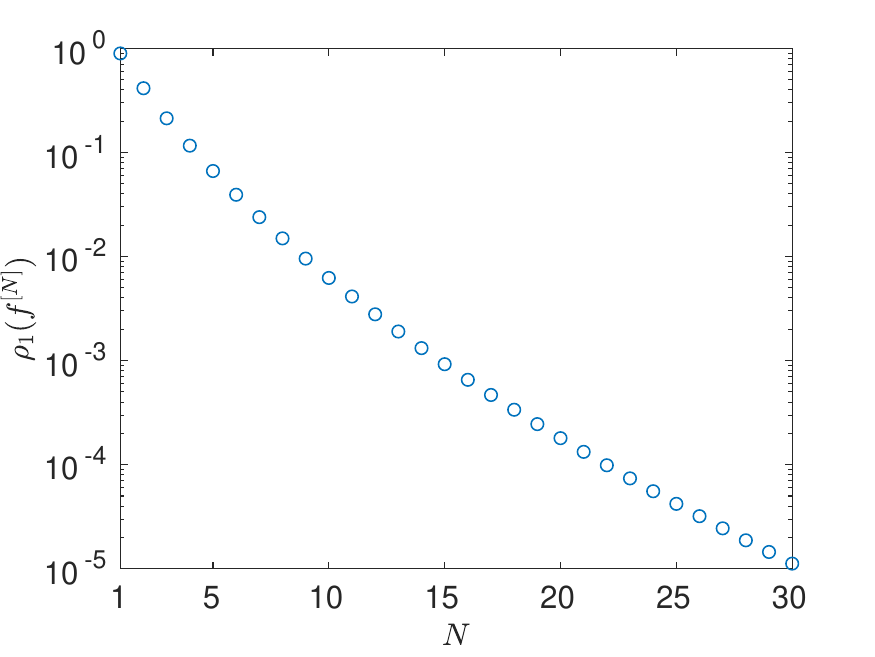}
\end{minipage}
\begin{minipage}{\textwidth}
\includegraphics[width=0.45\textwidth]{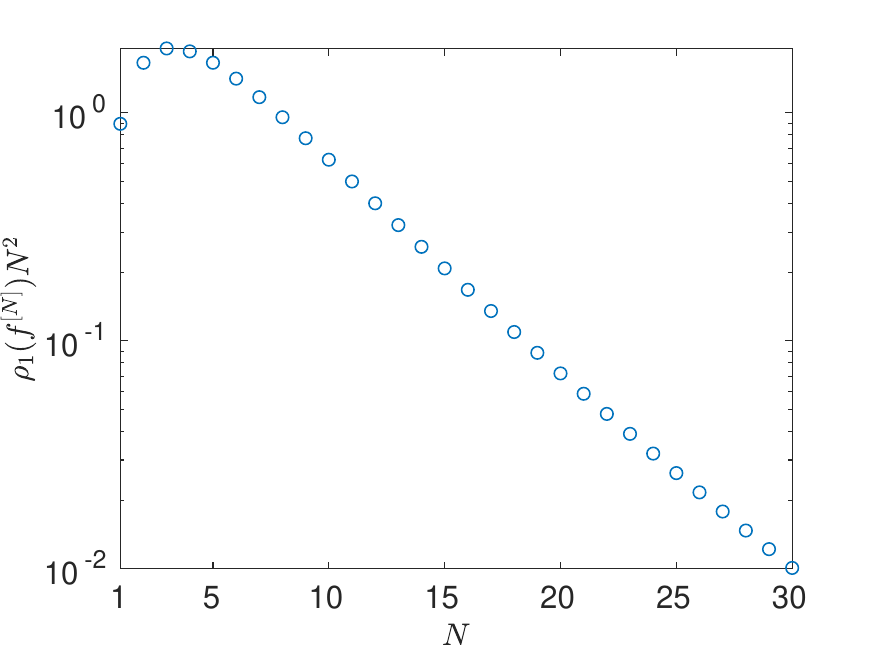}
\includegraphics[width=0.45\textwidth]{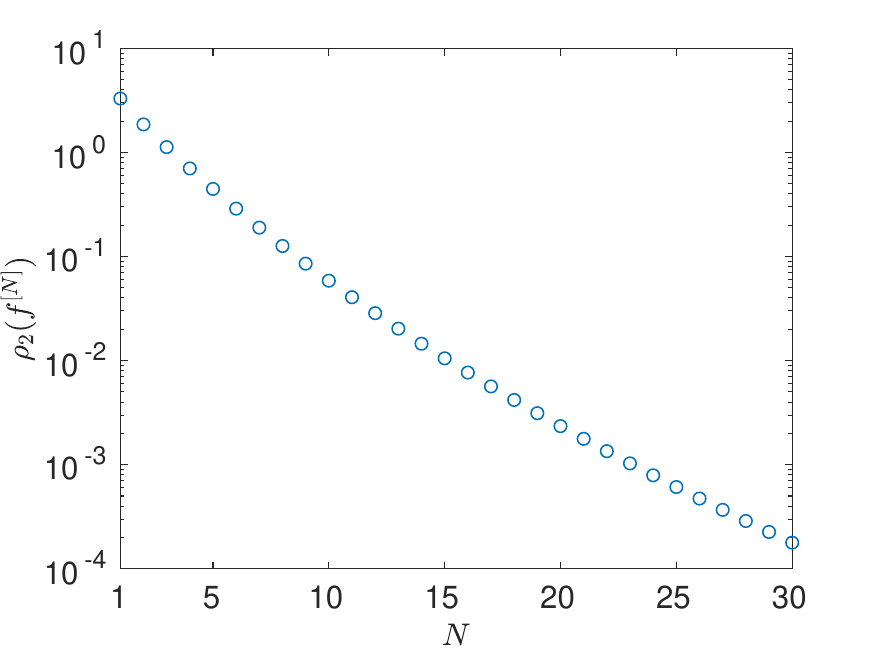}
\end{minipage}
\caption{Numerical experiments for test 1a. From top left: $\rho_0(f^{[N]})$, $\rho_1(f^{[N]})$, $\rho_1(f^{[N]})N^2$, and $\rho_2(f^{[N]})$ indicators of convergence vs $N$.}\label{fig:1}
\end{figure}

\begin{figure}
\begin{minipage}{\textwidth}
\includegraphics[width=0.45\textwidth]{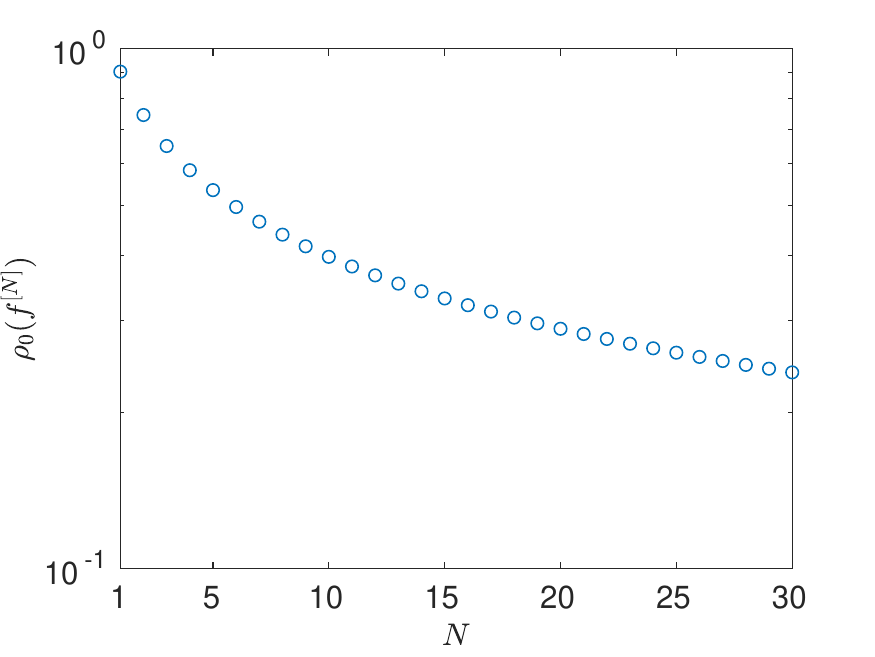}
\includegraphics[width=0.45\textwidth]{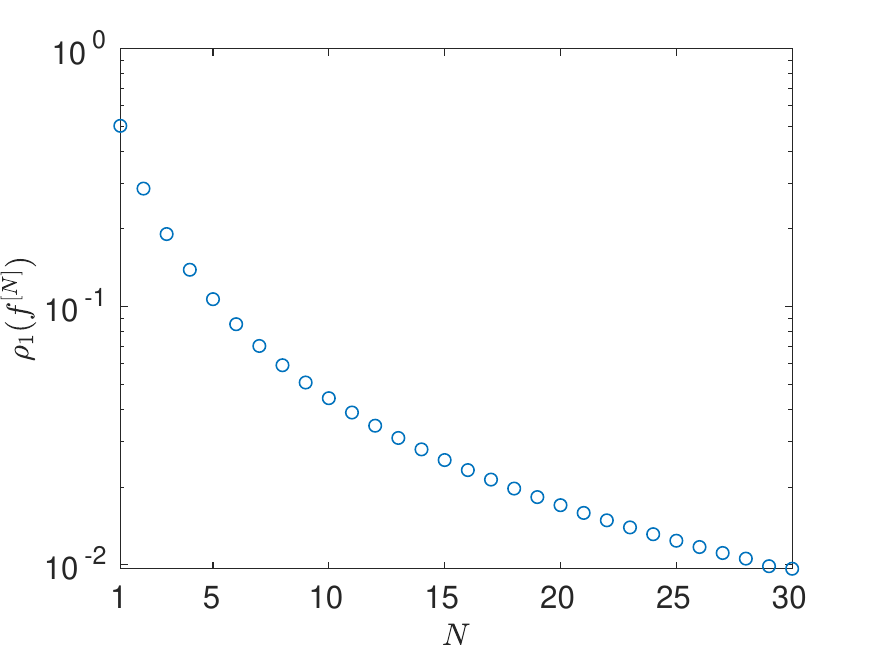}
\end{minipage}
\begin{minipage}{\textwidth}
\includegraphics[width=0.45\textwidth]{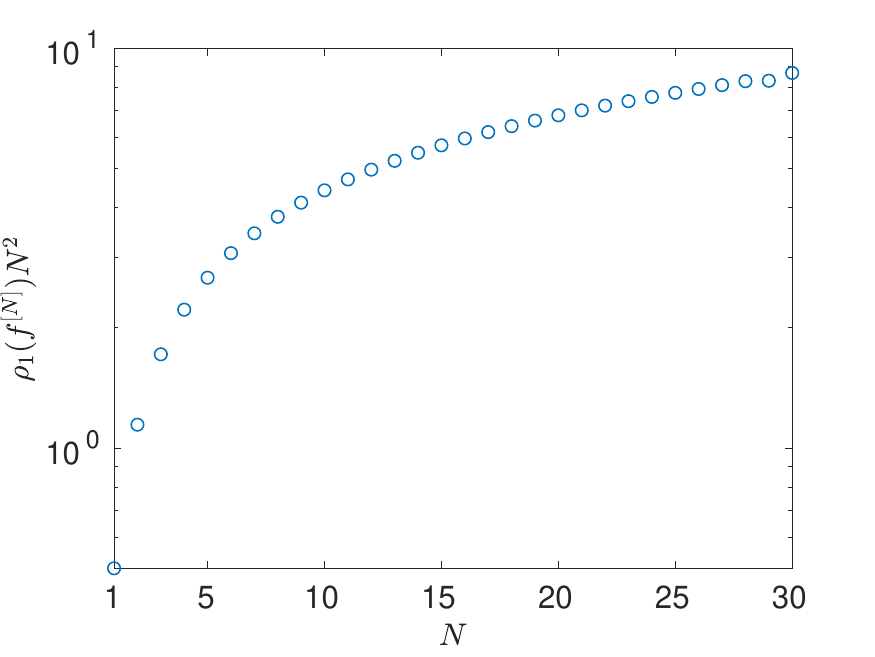}
\includegraphics[width=0.45\textwidth]{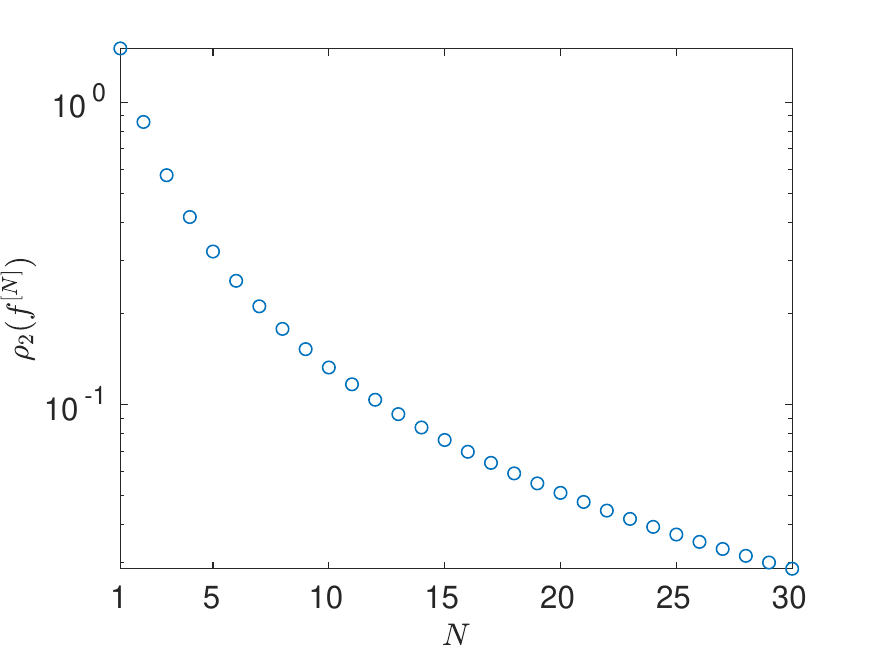}
\end{minipage}
\caption{Numerical experiments for test 1b. From top left: $\rho_0(f^{[N]})$, $\rho_1(f^{[N]})$, $\rho_1(f^{[N]})N^2$, and $\rho_2(f^{[N]})$ indicators of convergence vs $N$.}\label{fig:2}
\end{figure}

\begin{figure}
\begin{minipage}{\textwidth}
\includegraphics[width=0.45\textwidth]{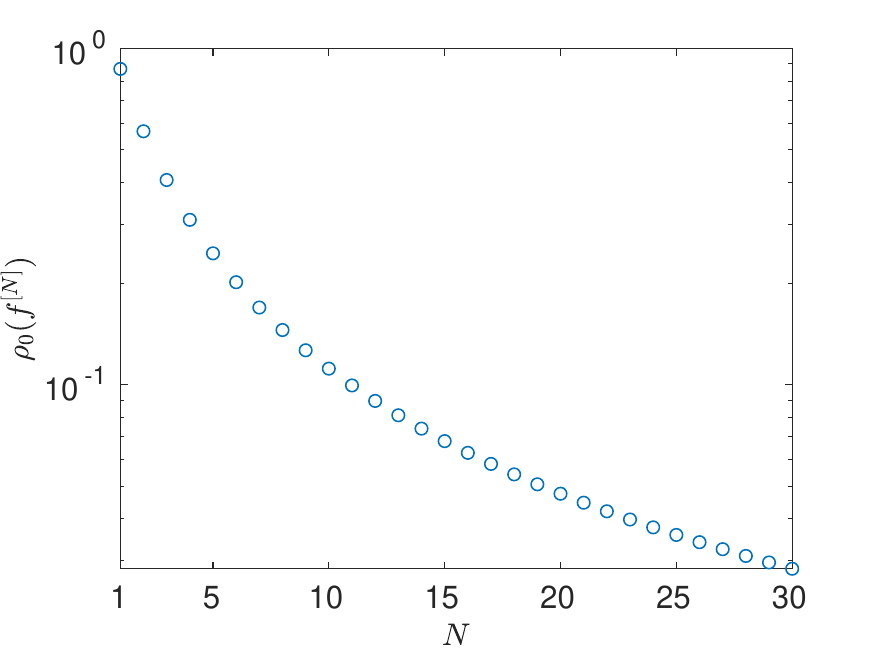}
\includegraphics[width=0.45\textwidth]{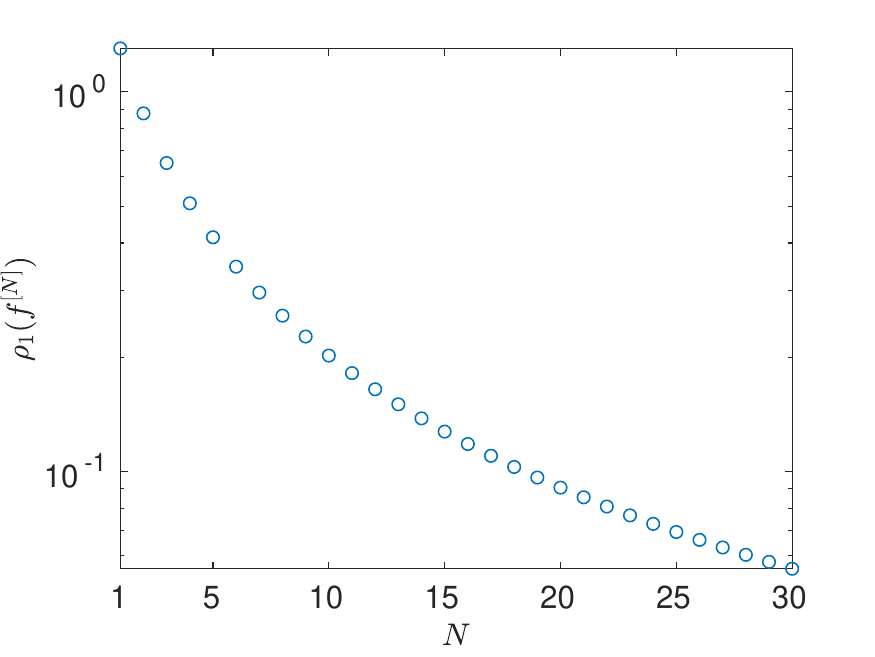}
\end{minipage}
\begin{minipage}{\textwidth}
\includegraphics[width=0.45\textwidth]{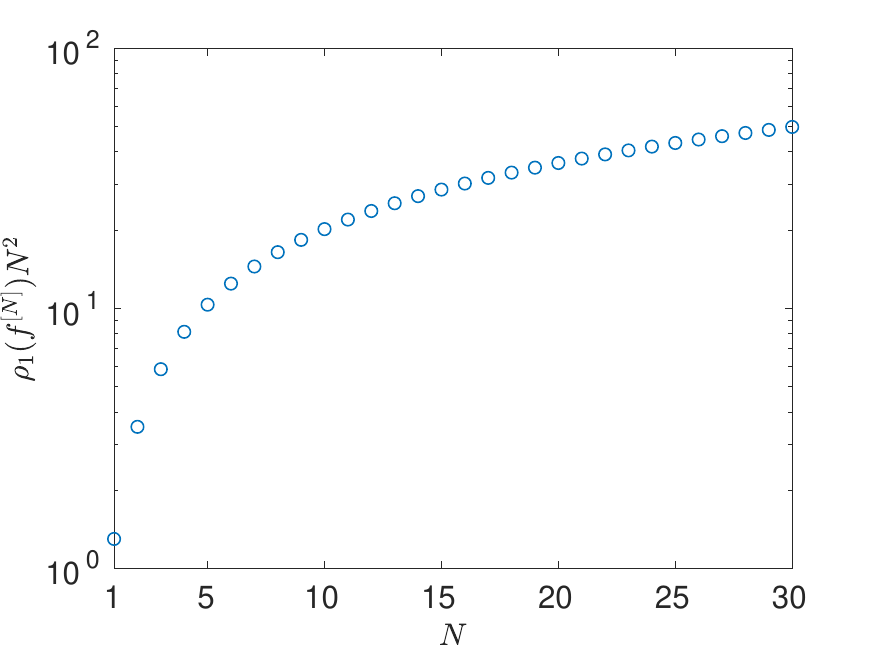}
\includegraphics[width=0.45\textwidth]{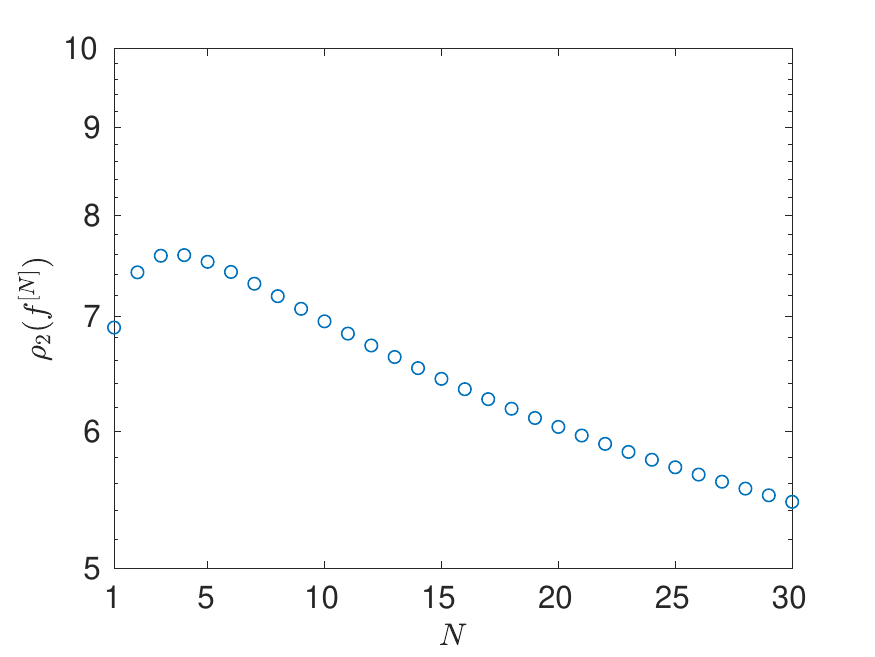}
\end{minipage}
\caption{Numerical experiments for test 2a. From top left: $\rho_0(f^{[N]})$, $\rho_1(f^{[N]})$, $\rho_1(f^{[N]})N^2$, and $\rho_2(f^{[N]})$ indicators of convergence vs $N$.}\label{fig:3}
\end{figure}

\begin{figure}
\begin{minipage}{\textwidth}
\includegraphics[width=0.45\textwidth]{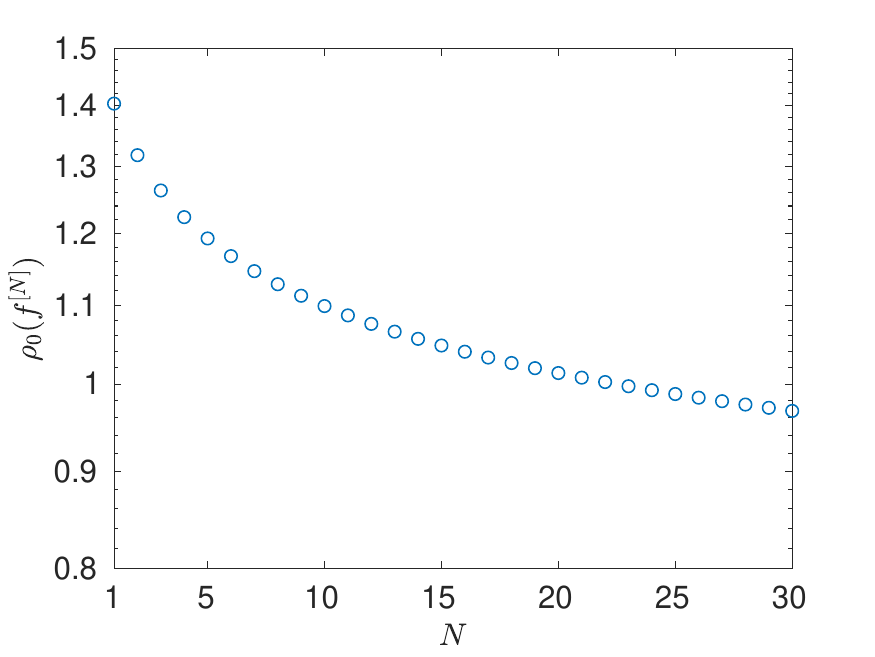}
\includegraphics[width=0.45\textwidth]{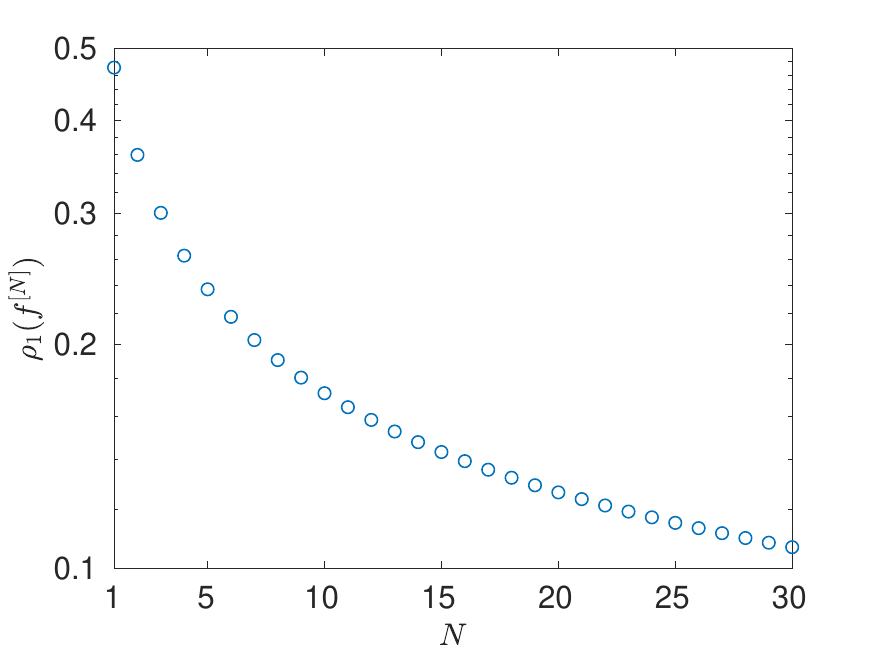}
\end{minipage}
\begin{minipage}{\textwidth}
\includegraphics[width=0.45\textwidth]{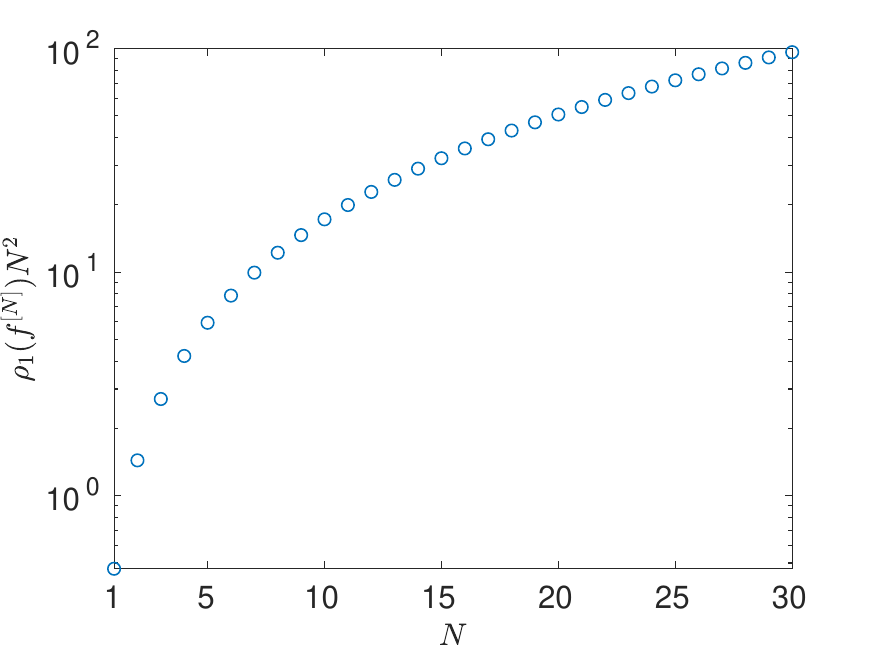}
\includegraphics[width=0.45\textwidth]{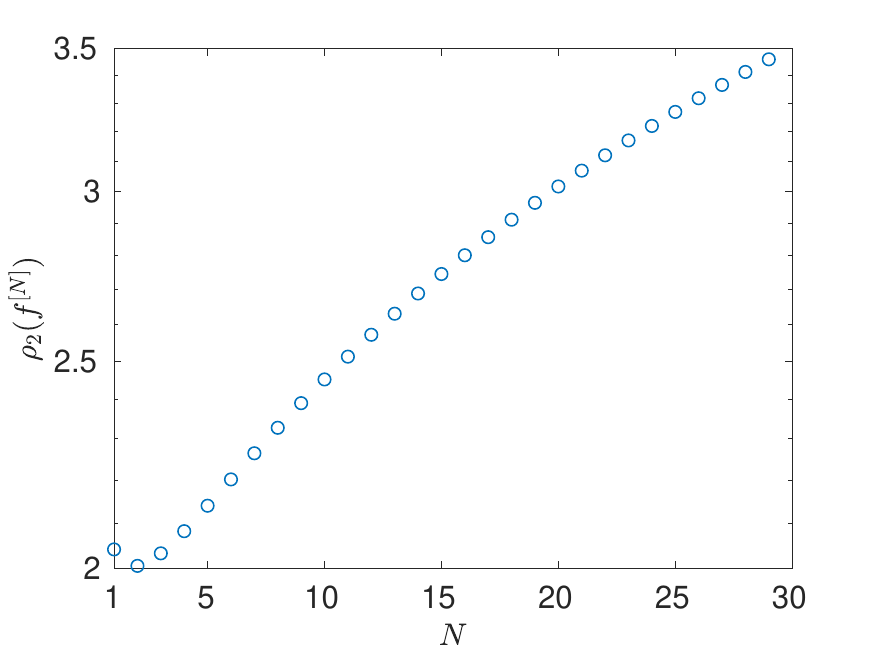}
\end{minipage}
\caption{Numerical experiments for test 2b. From top left: $\rho_0(f^{[N]})$, $\rho_1(f^{[N]})$, $\rho_1(f^{[N]})N^2$, and $\rho_2(f^{[N]})$ indicators of convergence vs $N$.}\label{fig:4}
\end{figure}

Both tests 1a and 1b reveal that the iterates not only converge in the sense of the error and of the energy norm as predicted by Theorem \ref{thm:main}, but also in the residual sense (not covered by Theorem \ref{thm:main}). Of course in retrospect the error's vanishing in test 1a is consistent with the residual's vanishing, owing to the boundedness of $A^{-1}$ in tests 1: indeed, obviously, 
\[
 \rho_0(f^{[N]})\;=\;\|f^{[N]}-f\big\|^2\;\leqslant\;\|A^{-1}\|_{\mathrm{op}}^2\|Af^{[N]}-g\big\|^2\;\leqslant\;\|A^{-1}\|_{\mathrm{op}}^2\,\,\rho_2(f^{[N]})\,.
\]
In addition, the classical Nemirovskiy-Polyak convergence rate for the energy norm is not violated in test 1a (Figure~\ref{fig:1}), whereas it appears to be violated in test 1b (Figure~\ref{fig:2}), where $A$ does not have a bounded inverse.

Heuristically, the slower vanishing rate of $\rho_0$, $\rho_1$, and $\rho_2$ in test 1b is due to the presence of zero in the spectrum of $A=-\frac{\ud^2}{\ud x^2}$: as we are approximating $A^{-1}g$ with polynomials $p_N(A)g$, the approximation to the inverse with polynomials is hampered about the ``bad'' spectral point $\lambda=0$.

As opposite to tests 1a and 1b, we know from Lemma \ref{lem:anal-non-quasian}(ii) that tests 2a and 2b, represented respectively in Figure~\ref{fig:3} and \ref{fig:4}, are \emph{not covered} by our theoretical analysis, but for the fact that the quantities $\rho_0(f^{[N]})$ and $\rho_1(f^{[N]})$ are surely predicted to stay uniformly bounded in $N$ (Remark \ref{rem:unif-bdd}). Such uniform boundedness is confirmed numerically.

In test 2a, where $A$ has bounded inverse on the whole $\cH$, numerics indicate that $\rho_0(f^{[N]})\to 0$ and $\rho_1(f^{[N]})\to 0$ as $N\to\infty$. That provides some practical evidence that there exist non-quasi-analytic data $g$ that still display ``good behaviour'', i.e., convergence at suitable $\rho_\sigma$-level. This is completely compatible with our Theorem \ref{thm:main} and Corollary \ref{cor:main}: the use of (quasi-)analyticity that we made therein is solely localised in Proposition \ref{prop:properties_of_zeros}(v) in order to apply Carleman's criterion for the determinacy of the Hamburger moment problem, and that criterion is only a sufficient condition for the uniqueness of the $\nu_\xi$-measure. In fact, investigating the nebulous regime beyond quasi-analyticity would be of substantial relevance to understand what minimal assumptions on $g$ and $f^{[0]}$ guarantee the uniqueness of the $\nu_\xi$-measure (that surely fails in certain cases, as we saw in Proposition \ref{prop:generic-regime}), or at least the convergence of the unbounded CG algorithm.

In comparison, in test 2b (unbounded $A^{-1}$) the decay rates of $\rho_0$ and $\rho_1$ appear to be slower than the counterpart 2a and it is unclear whether there is an actual vanishing, beside the evident decreasing behaviour.

The residual $\rho_2$ looks clearly decreasing in test 2a, with insufficient numerical evidence for vanishing, though, and instead manifestly divergent in test 2b. Here the solution $f=(1+x^2)^{-1}$ is not localised as the Gaussian of tests 1a and 1b, and has instead a long tail at large distances: the intuition suggests that this feature affects the convergence at higher regularity levels.

In either test 2a and 2b numerics give definite evidence of \emph{violation} of the Nemirovskiy-Polyak convergence rate.

\section*{Acknowledgements}

We warmly thank Prof Nemirovskiy for providing us with very useful comments and clarifications on his work \cite{Nemirovskiy-Polyak-1985}. A preliminary version of the present work was presented and discussed at the 6th Najman Conference on Spectral Theory and Differential Equations that took place in Sveti Martin na Muri (Croatia) in September 2019: we thank the organisers for that opportunity.


\def\cprime{$'$}

\end{document}